\documentclass[12pt,twoside]{amsart}
\usepackage{latexsym,amsfonts,amsmath,amssymb,soul}
\usepackage{enumerate,soul}
\usepackage[titletoc]{appendix}
\usepackage[usenames, dvipsnames]{color}
\usepackage{tikz}
\numberwithin{equation}{section}

\newtheorem{Th}{Theorem}[section]
\newtheorem{Rem}[Th]{Remark}

\newtheorem{Lemma}[Th]{Lemma}
\newtheorem{Def}[Th]{Definition}
\newtheorem{Prop}[Th]{Proposition}
\newtheorem{Cor}[Th]{Corollary}

\catcode`\@=11

\renewcommand{\section}%
   {\setcounter{equation}{0}\@startsection {section}{1}{\z@}{-3.5ex plus -1ex
  minus -.2ex}{2.3ex plus .2ex}{\Large\bf}}

\def\id{\mathop{\rm id}\nolimits}

\def\ds{\displaystyle}
\def\R{\mathbb R}
\def\C{\mathbb C}
\def\N{\mathbb N}

\newcommand{\D}{\mathcal{D}}
\newcommand{\E}{\mathcal{E}}

\newcommand{\M}{\mathcal{M}}
\newcommand{\Sch}{\mathcal{S}}

\newcommand{\bfM}{\mathbf{M}}
\newcommand{\bfN}{\mathbf{N}}
\newcommand{\bfc}{\mathbf{c}}
\newcommand{\bfW}{\mathbf{W}}

\newcommand{\afrac}[2]{\genfrac{}{}{0pt}{1}{#1}{#2}}

\newcommand{\beqsn}{\arraycolsep1.5pt\begin{eqnarray*}}
\newcommand{\eeqsn}{\end{eqnarray*}\arraycolsep5pt}
\newcommand{\beqs}{\arraycolsep1.5pt\begin{eqnarray}}
\newcommand{\eeqs}{\end{eqnarray}\arraycolsep5pt}

\title{Nuclear global spaces of ultradifferentiable functions in the matrix weighted setting}

\author[Boiti]{Chiara Boiti}
\address{
Dipartimento di Matematica e Informatica \\Universit\`a di Ferrara\\
Via Ma\-chia\-vel\-li n.~30\\
I-44121 Ferrara\\
Italy}
\email{chiara.boiti@unife.it}

\author[Jornet]{David Jornet}
\address{
Instituto Universitario de Matem\'atica Pura y Aplicada IUMPA\\
Universitat Po\-li\-t\`ecni\-ca de Val\`encia\\
Camino de Vera, s/n\\
E-46071 Valencia\\
Spain}
\email{djornet@mat.upv.es}

\author[Oliaro]{Alessandro Oliaro}
\address{Dipartimento di Matematica\\ Universit\`a di Torino\\
 Via Carlo Alberto n.~10\\ I-10123 Torino\\ Italy}
 \email{alessandro.oliaro@unito.it}

\author[Schindl]{Gerhard Schindl}
\address{Fakult\"at f\"ur Mathematik\\ Universit\"at Wien\\
Oskar-Morgenstern-Platz n.~1\\ A-1090 Wien\\ Austria}
 \email{gerhard.schindl@univie.ac.at}

\begin{document}


\begin{abstract}
We prove that the Hermite functions are an absolute Schauder basis for many global weighted spaces of ultradifferentiable functions in the matrix weighted setting and we determine also the corresponding coefficient spaces, thus extending previous work by Langenbruch.~As a consequence we give very general conditions for these spaces to be nuclear.~In particular, we obtain the corresponding results for spaces defined by weight functions.  
\end{abstract}

\maketitle

\newcount\minutes
 \newcount\hour
 \newcount\minutea
 \newcount\minuteb
 \minutes=\time
  \divide\minutes by 60
  \hour=\minutes
\minutes=\time
 \multiply \hour by 60
  \advance \minutes by -\hour
 \divide \hour by 60
 \minuteb=\minutes
 \divide\minuteb by 10
\minutea=\minuteb
\multiply \minuteb by 10
\advance \minutes by  -\minuteb
\minuteb=\minutes
\date{\today \space \number\hour:\number\minutea \number\minuteb }
 \markboth{\today \space \number\hour:\number\minutea \number\minuteb }{\today \space \number\hour:\number\minutea \number\minuteb }

\markboth{\sc Nuclear global spaces of ultradifferentiable functions in the matrix weighted setting}
 {\sc C.~Boiti, D.~Jornet, A.~Oliaro and G.~Schindl}

\section{Introduction}

The systematic study of nuclear locally convex spaces began in 1951 with the fundamental dissertation of A.~Grothendieck~\cite{Gr} to classify those infinite dimensional locally convex spaces which are not normed, suitable for mathematical analysis. Among the properties of a nuclear space, the existence of a Schwartz kernel for a continuous linear operator on the space is of crucial importance for the theory of linear partial differential operators. In our setting of ultradifferentiable functions, this fact helps, for instance, to study the behaviour (propagation of singularities or wave front sets) of a differential or pseudodifferential operator when acting on a distribution. See, for example, \cite{AJ,BJO-Gabor,fgj1,fgj2,NR,RW} and the references therein.

Since the middle of the last century several authors have studied the topological structure of global spaces of ultradifferentiable functions and, in particular, when the spaces are nuclear. See \cite{Mi}, or the book \cite{GV}. More recently,  the first three authors in \cite{BJO-Rodino} used the isomorphism established by Langenbruch~\cite{L} between global spaces of ultradifferentiable functions in the sense of Gel'fand and Shilov~\cite{GS} and some sequence spaces to see that under the condition that appears in  \cite[Corollary 16(3)]{BMM} on the weight function $\omega$ (in the sense of \cite{BMT})  the space $\Sch_{(\omega)}(\R^d)$ of rapidly decreasing ultradifferentiable functions of Beurling type in the sense of Bj\"orck~\cite{Bj} is nuclear. However, there was the restriction that the powers of the logarithm were not allowed as admissible weight functions. Later, the authors of the present work  proved in \cite{BJOS} that $\Sch_{(\omega)}(\R^d)$ is nuclear for any weight function satisfying $\log(t)=O(\omega(t))$  and $\omega(t)=o(t)$ as $t$ tends to infinity. The techniques used in \cite{BJOS} come especially from the field of time-frequency analysis and a mixture of ideas from \cite{BJO-Gabor,G,GZ,RW}. In both \cite{BJO-Rodino} and \cite{BJOS} we use (different) isomorphisms between that space $\Sch_{(\omega)}(\R^d)$ and some sequence space and prove that $\Sch_{(\omega)}(\R^d)$ is nuclear by an application of the Grothendieck-Pietsch criterion~\cite[Theorem 28.15]{MV}. Very recently, Debrouwere, Neyt and Vindas~\cite{DNV1,DNV2} (cf.~\cite{Kru} for related results about local spaces), using different techniques have extended our previous results in a very general framework. In \cite{DNV1} they characterize when mixed spaces of Bj\"orck~\cite{Bj} of Beurling type or of Roumieu type are nuclear under very mild conditions on the weight functions. In \cite{DNV2}, using weight matrices in the sense of \cite{RS} the same authors characterize the nuclearity of generalized Gel'fand-Shilov classes which extend their previous work \cite{DNV1} and treat also many other mixed classes defined by sequences.

The aim of the present paper is twofold. On the one hand, we extend the work of Langenbruch~\cite{L} to the matrix weighted setting in the sense of \cite{RS,S-PhDtesi}. In particular, we prove that the Hermite functions are a Schauder basis of many global weighted spaces of ultradifferentiable functions. Moreover, we determine the coefficient spaces corresponding to this Hermite expansion (Theorem~\ref{th41G}). These results are applied to spaces defined by weight functions $\Sch_{[\omega]}(\R^d)$, being $[\omega]=(\omega)$ (Beurling setting) or $[\omega]=\{\omega\}$ (Roumieu setting). Hence, we extend part of the previous work of Aubry~\cite{A} to the several variables case. 
As a consequence we extend to a very general situation  our previous study \cite{BJO-Rodino,BJOS}  about the nuclearity of the space $\Sch_{(\omega)}(\R^d)$   to global spaces of ultradifferentiable functions defined by weight matrices (Corollary~\ref{lemma57Gend}). An application to particular matrices gives that $\Sch_{(\omega)}(\R^d)$ is nuclear when $\omega(t)=o(t^2)$ as $t$ tends to infinity. Similarly we also prove the analogous result for the Roumieu setting, namely that $\Sch_{\{\omega\}}(\R^d)$ is nuclear when $\omega(t)=O(t^2)$ as $t$ tends to infinity (see Corollary~\ref{cor56G} for both results). For weights of the form $\omega(t)=\log^\beta(1+t)$ with $\beta>1$ our results hold and, hence, we generalize the results of \cite{L} to spaces that could not be treated there since, as is easily deduced from \cite[Example 20]{BMM}, $\Sch_{[(M_p)_p]}(\R)\neq \Sch_{[\omega]}(\R)$ for any sequence of positive numbers $(M_p)_{p\in\N}$ in the sense of \cite{K} (see Remark~\ref{remarkseparating}). We do not treat here the classical case $\omega(t)=\log(1+t)$, for which $\Sch_{(\omega)}(\R)=\Sch(\R)$, the Schwartz class, because in this case infinitely many entries of our weight matrices are not well defined. However, the results presented here are already well known for the Schwartz class.

 The classes of functions treated in \cite{DNV2} are in general different from ours. In fact, here we consider spaces of functions $f$ that are bounded in the following sense: for some (or any) $h>0$, there is $C>0$ such that for all $x\in\R^d$ and every multi-indices $\alpha$ and $\beta$  we have  
 $$
 (A)\qquad\qquad |x^\alpha \partial^\beta f(x)|\le C h^{|\alpha+\beta|}M_{\alpha+\beta}.
 $$
 And we pass to the matrix setting for the multi-sequence $(M_\alpha)_\alpha$, i.e. we make $M_\alpha^\lambda$  depend also on a parameter $\lambda>0$ (see the precise definition in the next section).
 In \cite{DNV2}, the authors consider spaces of functions $f$ bounded in the following sense: there is $C>0$ such that for all $x\in\R^d$ and every multi-index $\beta$ they have
 $$
 (B)\qquad \qquad |w(x) \partial^\beta f(x)|\le C M_\beta,
 $$
 where $w$ is a positive continuous function. They pass to the matrix setting by making $M_\beta^\lambda$ and $w^\lambda$ depend on the same parameter $\lambda>0$. Hence, taking unions (Roumieu setting) or intersections (Beurling setting) in $\lambda$ in the situation $(A)$ gives different classes of functions than in the situation $(B)$ in general. 
  On the other hand, it is a very difficult problem to determine when the classes treated in this work are non-trivial, a question not considered in \cite{DNV1,DNV2}. We characterize in a very general way (Propositions~\ref{cor36G} and \ref{lemma37G}) when the Hermite functions are contained in our classes and this fact is closely related to classes being non-trivial. Indeed, we can deduce from our results that, in the Beurling setting, the space $\Sch_{(\omega)}(\R^d)$ contains the Hermite functions if and only if $\omega(t)=o(t^2)$ as $t$ tends to infinity (Corollary~\ref{remopiccolo}). However, it is not difficult to see from the uncertainty principle \cite[Theorem]{Hi} that $\Sch_{(\omega)}(\R^d)=\{0\}$ when $t^2=O(\omega(t))$ as $t$ tends to infinity. In the same way, in the Roumieu case the space $\Sch_{\{\omega\}}(\R^d)$ contains the Hermite functions if and only if $\omega(t)=O(t^2)$ as $t$ tends to infinity (Corollary~\ref{remopiccolo}), but again from \cite[Theorem]{Hi}  we can deduce $\Sch_{\{\omega\}}(\R^d)=\{0\}$ when $t^2=o(\omega(t))$ as $t$ tends to infinity. For more information on the uncertainty principle for $\Sch_*(\R^d)$ being $\ast=(\omega)$ or $\ast=\{\omega\}$ see the nice introduction to the paper of Aubry~\cite{A} and the references therein.  Moreover, our classes are well adapted for Fourier transform (Corollary~\ref{Fourier-iso}). We should also mention that throughout this paper we assume,  on the multi-sequence $(M_\alpha)_\alpha$, that $(M_{\alpha})^{1/|\alpha|}$ tends to infinity when  $|\alpha|$ tends to infinity, which is stronger than the condition $\inf_{\alpha\in \N^d_0}(M_\alpha/M_0)^{1/|\alpha|}>0$ considered in \cite{DNV2}. The reason is that it is not clear how the results read when the associated function is infinite (see Remark~\ref{omegafinite}).

The paper is organized as follows: in the next section we give some necessary definitions, in Section~\ref{sec2prime} we introduce the classes under study in the matrix weighted setting and establish the analogous conditions to \cite{L} to determine in Section~\ref{sec3} when the Hermite functions belong to our classes. In Section~\ref{sec4} we introduce the suitable matrix sequence spaces and prove that they are isomorphic to our classes, which is the fundamental tool to see that our spaces are nuclear. We finally apply these results to the particular case of spaces defined by weight functions in Section~\ref{sec5}.

\section{Preliminaries}
\label{sec2}

We briefly recall from \cite{K} those basic notions about sequences $\bfM=(M_p)_{p\in\N_0}$, for $\N_0:=\N\cup\{0\}$, that we need in what follows. A sequence $(M_p)_p$ is called {\em normalized} if $M_0=1$. For a normalized sequence
$\bfM=(M_p)_p$ the {\em associated function} is denoted by
\beqs
\label{assofunc}
\omega_{\bfM}(t)=\sup_{p\in \N_0}\log\frac{|t|^p}{M_p},\qquad t\in\R.
\eeqs
We say that $(M_p)_p$ satisfies the {\em logarithmic convexity} condition $(M1)$
of \cite{K} if
\beqs
\label{M1}
M_p^2\leq M_{p-1}M_{p+1},\qquad p\in\N.
\eeqs
The following lemma is well known (see Lemmas 2.0.6 and 2.0.4
of \cite{S-tesi} for a proof).
\begin{Lemma}
\label{lemmatrivial}
Let $(M_p)_{p\in\N_0}$ be a normalized sequence satisfying \eqref{M1}. Then
\begin{itemize}
\item[(a)]
$M_jM_k\leq M_{j+k}$ for all $j,k\in\N_0$;
\item[(b)]
$p\mapsto(M_p)^{1/p}$ is increasing;
\item[(c)]
$\liminf_{p\to+\infty}(M_p)^{1/p}>0$.
\end{itemize}
\end{Lemma}


From Lemma~\ref{lemmatrivial}(c) and \cite[Prop. 3.2]{K}, we have that a normalized sequence $\bfM=(M_p)_p$ satisfies \eqref{M1} if and only if
\beqs
\label{33k}
M_p=\sup_{t>0}\frac{t^p}{\exp\omega_{\bfM}(t)},\qquad p\in\N_0.
\eeqs

We say that $(M_p)_p$ satisfies the {\em stability under differential operators}
condition $(M2)'$ of \cite{K} if
\beqs
\label{M2'}
\exists A,H\geq1\ \forall p\in\N_0:\quad M_{p+1}\leq AH^pM_p,
\eeqs
and
$(M_p)_p$ satisfies the stronger {\em moderate growth}
condition $(M2)$ of \cite{K} if
\beqs
\label{M2}
\exists A\geq1\ \forall p,q\in\N_0:\quad M_{p+q}\leq A^{p+q}M_pM_q.
\eeqs

%
%

The following lemma extends \cite[Proposition 3.4]{K} for two sequences. We give the proof for the convenience of the reader.
\begin{Lemma}
\label{lemma23G}
Let $\bfM=(M_p)_{p\in\N_0}$
and $\bfN=(N_p)_{p\in\N_0}$ be two normalized sequences satisfying \eqref{M1}.
Then the following conditions are equivalent:
\begin{itemize}
\item[(i)]
$\exists A\geq1\ \forall p\in\N_0: \quad M_{p+1}\leq A^{p+1}N_p$.
\item[(ii)]
$\exists A\geq1, B>0\ \forall t>0: \quad\omega_{\bfN}(t)+\log t\leq\omega_{\bfM}(At)+B$.
\end{itemize}
\end{Lemma}

\begin{proof}
If $(i)$ is satisfied, then, for all $t>0$,
\beqsn
te^{\omega_\bfN(t)}=t\sup_{p\in\N_0}\frac{t^p}{N_p}\leq
\sup_{p\in\N_0}\frac{(At)^{p+1}}{M_{p+1}}
\leq\sup_{p\in\N_0}\frac{(At)^p}{M_p}=e^{\omega_\bfM(At)}.
\eeqsn
Conversely, if $(ii)$ holds, then, by \eqref{33k},
\beqsn
N_p=&&\sup_{t>0}\frac{t^p}{\exp\omega_\bfN(t)}\geq\sup_{t>0}
\frac{t^{p+1}}{e^B\exp\omega_\bfM(At)}\\
=&&e^{-B}\sup_{s>0}\frac{(s/A)^{p+1}}{\exp\omega_\bfM(s)}
=\frac{e^{-B}}{A^{p+1}}M_{p+1}.
\eeqsn
\end{proof}
Now, we consider sequences $\bfM=(M_\alpha)_{\alpha\in\N_0^d}$ of
positive real numbers for multi-indices $\alpha\in\N_0^d$. 
As in the one-dimensional case, we say that $(M_\alpha)_{\alpha\in\N_0^d}$ is normalized if $M_0=1$. 
We recall condition (3.7) of \cite{L}
\beqs
\label{37L}
\exists A\geq1\ \forall\alpha,\beta\in\N_0^d: \quad M_\alpha
M_\beta\leq A^{|\alpha+\beta|}M_{\alpha+\beta}.
\eeqs
Condition \eqref{M2'} takes in this setting the form (see \cite[$(2.1)$]{L})
\beqs
\label{M2'aniso}
\exists A\geq1\ \forall \alpha\in\N_0^d, 1\leq i\leq d:\quad M_{\alpha+e_j}\leq A^{|\alpha|+1}M_{\alpha},
\eeqs
and \eqref{M2} turns into
\beqs
\label{M2aniso}
\exists A\geq1\ \forall \alpha,\beta\in\N_0^d: \quad M_{\alpha+\beta}\leq A^{|\alpha+\beta|}M_{\alpha}M_{\beta}.
\eeqs
Now, for $t\in\R^d$, we denote
\beqs
\label{4}
\N_{0,t}^d:=\{\alpha\in\N^d_0 : \alpha_j=0\ \text{if}\ t_j=0,\ j=1,\dots,d\}.
\eeqs
The {\em associated weight function} of a normalized  $\bfM=(M_\alpha)_{\alpha\in\N_0^d}$ is given by
\beqsn
\omega_\bfM(t)=\sup_{\alpha\in\N_{0,t}^d}\log\frac{|t^\alpha|}{M_\alpha},\;\;\;t\in\R^d,
\eeqsn
where by convention $0^0:=1$. Note that for a normalized sequence we have $\omega_{\bfM}(0)=0$.

%

 \begin{Rem}\label{omegafinite}{\rm 
As it has already been pointed out in the geometric construction in \cite[Chap.~I]{MA} for the one dimensional weight function (see \eqref{assofunc}), we have that $\omega_\bfM(t)<+\infty$ for all $t\in\R^d$ if and only if $\lim_{|\alpha|\rightarrow\infty}(M_{\alpha})^{1/|\alpha|}=+\infty$.


First, assume that $\omega_\bfM(t)<+\infty$ for all $t\in\R^d$. Hence for all $t=(t_1,\dots,t_d)\in\R^d$ satisfying $t_{\min}:=\min_{1\le j\le d}|t_j|\ge 1$ we have that there exists some $C$ (depending only on $t$) such that $\log\frac{|t^\alpha|}{M_\alpha}\le C$ for all $\alpha\in\N_0^d$. So $t_{\min}^{|\alpha|}\le|t_1^{\alpha_1}\cdots t_d^{\alpha_d}|=|t^{\alpha}|\le e^CM_{\alpha}$ for all $\alpha\in\N_0^d$ and now let $t_{\min}\rightarrow+\infty$.

Conversely, let $\lim_{|\alpha|\rightarrow\infty}(M_{\alpha})^{1/|\alpha|}=\infty$ and so for any $A>0$ large, we can find some $C>0$ large enough such that $A^{|\alpha|}\le CM_{\alpha}$. Since $|t^{\alpha}|\le|t|^{|\alpha|}$ for all $t\in\R^d$ and $\alpha\in\N_0^d$ we see that for any given $t\in\R^d$ we get $\frac{|t^\alpha|}{M_\alpha}\le\frac{|t|^{|\alpha|}}{M_{\alpha}}\le C$ for some $C>0$ and all $\alpha\in\N_0^d$.}
\end{Rem}

%

\begin{Lemma}
\label{lemma2}
Let $\bfM=(M_\alpha)_{\alpha\in\N_0^d}$. Then, for all
$h>0$ and $\alpha\in\N_0^d$,
\beqs
\label{13G}
M_\alpha h^{|\alpha|}\geq\sup_{t\in\R^d}|t^\alpha e^{-\omega_\bfM(t/h)}|.
\eeqs
\end{Lemma}

\begin{proof}
Fix $\alpha\in\N^d_0$ and $h>0$; we write $\R^d_\alpha:=\{t\in\R^d : t_j\neq 0\ \text{for}\ \alpha_j\neq 0,\ j=1,\dots,d\}$. Then for $t\in\R^d\setminus \R^d_\alpha$ we have $t^\alpha=0$, and so it is enough to prove that
\beqs
\label{44}
M_\alpha h^\alpha\geq \sup_{t\in\R^d_\alpha} |t^\alpha e^{-\omega_\bfM(t/h)}|.
\eeqs
We have
\beqsn
\frac{1}{\ds\sup_{t\in\R^d_\alpha}|t^\alpha e^{-\omega_\bfM(t/h)}|}=&&\inf_{t\in\R^d_\alpha}
\frac{e^{\omega_\bfM(t/h)}}{|t^\alpha|}
=\inf_{t\in\R^d_\alpha}\frac{\exp{\ds\sup_{\beta\in\N_{0,t}^d}\log\frac{\big|\left(\frac th\right)^\beta
\big|}{M_\beta}}}{|t^\alpha|}=\inf_{t\in\R^d_\alpha}
\frac{1}{|t^\alpha|}\sup_{\beta\in\N_{0,t}^d}\frac{\big|\left(\frac th\right)^\beta\big|}{
M_\beta};
\eeqsn
observe that $\alpha\in \N^d_{0,t}$ and so, choosing $\beta=\alpha$, we get
\beqsn
\frac{1}{\ds\sup_{t\in\R^d_\alpha}|t^\alpha e^{-\omega_\bfM(t/h)}|}\geq\inf_{t\in\R^d_\alpha}\frac{|t^\alpha|}
{|t^\alpha| h^{|\alpha|}M_\alpha}
=\frac{1}{h^{|\alpha|}M_\alpha},
\eeqsn
which proves \eqref{44}, and then the proof is complete.

Note that if $\omega_\bfM(t/h)=+\infty$, then \eqref{13G} is clear and so we could restrict in the estimates above to all $t\in\R^d_\alpha$ such that $\omega_\bfM(t/h)$ is finite.
\end{proof}

In the following we use two normalized sequences as above  $\bfM=(M_\alpha)_{\alpha\in\N_0^d}$
and $\bfN=(N_\alpha)_{\alpha\in\N_0^d}$ and we compare them  in the sense:
\beqsn
\bfM\leq\bfN\quad\mbox{if}\quad M_\alpha\leq N_\alpha, \quad \alpha\in\N_0^d.
\eeqsn
This clearly implies 
\beqsn
\omega_\bfN(t)\leq\omega_\bfM(t),\quad t\in\R^d.
\eeqsn
In~\cite{L}, Langenbruch uses his condition (1.2) to prove that the Hermite functions belong to the spaces considered there. In the present paper we need, for the same reason, a mixed condition that involves two sequences:
\beqs
\label{12L1R}
\exists H,C,B>0\ \forall \alpha,\beta\in\N_0^d:\quad
\alpha^{\alpha/2}M_\beta\leq BC^{|\alpha|}H^{|\alpha+\beta|}N_{\alpha+\beta}.
\eeqs

\begin{Rem}\label{omegafinite1}
\begin{em}
Condition \eqref{12L1R} yields that $\lim_{|\alpha|\rightarrow\infty}(N_{\alpha})^{1/|\alpha|}=+\infty$. 
Indeed, since by convention $0^0=1$ and by definition  $\alpha^{\alpha/2}:=\alpha_1^{\alpha_1/2}\cdots\alpha_d^{\alpha_d/2}$, from \eqref{12L1R} with 
$\beta=0$ we get, for $|\alpha|_\infty:=\max_{1\leq j\leq d}\alpha_j$,
\beqsn
N_\alpha^{1/|\alpha|}\geq&& B^{-\frac{1}{|\alpha|}}C^{-1}H^{-1}(\alpha^{\alpha/2})^{1/|\alpha|}
=B^{-\frac{1}{|\alpha|}}C^{-1}H^{-1}(\alpha_1^{\alpha_1/2}\cdots\alpha_d^{\alpha_d/2})^{1/|\alpha|}\\
\geq&&B^{-\frac{1}{|\alpha|}}C^{-1}H^{-1}(|\alpha|_\infty^{|\alpha|_\infty/2})^{1/|\alpha|}
\geq B^{-\frac{1}{|\alpha|}}C^{-1}H^{-1}\left(\frac{|\alpha|}{d}\right)^{\frac{1}{2d}}\to+\infty.
\eeqsn
\end{em}
\end{Rem}
\section{Global ultradifferentiable functions in the matrix weighted setting}
\label{sec2prime}
In this section we consider matrices of normalized sequences $(M^{(\lambda)}_\alpha)_{\lambda>0,\alpha\in\N_0^d}$ of real positive numbers:
\beqs
\label{defcalM}
\qquad \M:=\{(\bfM^{(\lambda)})_{\lambda>0}:\ \bfM^{(\lambda)}=
(M^{(\lambda)}_\alpha)_{\alpha\in\N_0^d},\
M^{(\lambda)}_0=1, \
\bfM^{(\lambda)}\leq\bfM^{(\kappa)}\,\mbox{for all}\,0<\lambda\leq\kappa\}.
\eeqs
We call $\M$ a {\em weight matrix} and consider matrix
weighted global ultradifferentiable functions of Roumieu type defined as follows
(from now on $\|\cdot\|_\infty$ denotes the 
supremum norm):
\beqsn
\Sch_{\{\bfM\}}:=&&\Big\{f\in C^\infty(\R^d):\  \exists C,h>0, \ \
\|f\|_{\infty,\bfM,h}:=\sup_{\alpha,\beta\in\N_0^d}
\frac{\|x^\alpha\partial^\beta f\|_\infty}{h^{|\alpha+\beta|}M_{\alpha+\beta}}\leq C\Big\},\\
\Sch_{\{\M\}}:=&&
\bigcup_{\lambda>0}\Sch_{\{\bfM^{(\lambda)}\}}
=\{f\in C^\infty(\R^d):\  \exists C,h,\lambda>0,\ \
\|f\|_{\infty,\bfM^{(\lambda)},h}\leq C\},
\eeqsn
endowed with the inductive limit topology (which may be thought countable if we take $\lambda,h\in\N$). For the Beurling setting, similarly we put:
\beqsn
\Sch_{(\bfM)}:=&&\{f\in C^\infty(\R^d):\  \forall h>0\ \exists C_h>0, \ \
\|f\|_{\infty,\bfM,h}\leq C_h\},\\
\Sch_{(\M)}:=&&
\bigcap_{\lambda>0}\Sch_{(\bfM^{(\lambda)})}
=\{f\in C^\infty(\R^d):\  \forall h,\lambda>0\ \exists C_{\lambda,h}>0,\ \
\|f\|_{\infty,\bfM^{(\lambda)},h}\leq C_{\lambda,h}\},
\eeqsn
endowed with the projective limit topology (countable for $\lambda^{-1},h^{-1}\in\N$).

%

Now we consider different conditions on the weight matrices that we use following the lines of \cite{L}. The next basic condition extends (1.2) of \cite{L} in the Roumieu case and is needed to show that the Hermite functions belong to $\Sch_{\{\M\}}$ (see Proposition~\ref{lemma37G}):
\beqs
\label{12L2R}
\qquad\forall\lambda>0\ \exists\;\kappa\geq\lambda, B,C,H>0\ \forall\alpha,\beta\in\N_0^d: \qquad
\alpha^{\alpha/2}M^{(\lambda)}_\beta\leq BC^{|\alpha|}H^{|\alpha+\beta|}
M^{(\kappa)}_{\alpha+\beta}.
\eeqs
The analogous condition to \eqref{12L2R} in the Beurling case, which is needed to show that the Hermite functions belong to $\Sch_{(\M)}$ is the following (see Proposition~\ref{lemma37G}):
\begin{equation}
\label{12L2B}
\begin{split}
&\forall\;\lambda>0\ \exists\;0<\kappa\leq\lambda, H>0\ \forall C>0\ \exists B>0\  \forall\alpha,\beta\in\N_0^d:\\
&\alpha^{\alpha/2}M^{(\kappa)}_\beta\leq BC^{|\alpha|}H^{|\alpha+\beta|}
M^{(\lambda)}_{\alpha+\beta}.
\end{split}
\end{equation}

\begin{Rem}\label{omegafinite2}
\begin{em}
Similarly, as commented in Remark \ref{omegafinite1} for \eqref{12L1R}, property \eqref{12L2R} (property \eqref{12L2B}) yields that $\lim_{|\alpha|\rightarrow\infty}(M^{(\kappa)}_{\alpha})^{1/|\alpha|}=+\infty$ for some $\kappa>0$, and hence for all $\kappa'\ge\kappa$ ($\lim_{|\alpha|\rightarrow\infty}(M^{(\lambda)}_{\alpha})^{1/|\alpha|}=+\infty$ for all $\lambda>0$).
\end{em}
\end{Rem}

We also need to extend condition (3.7) of \cite{L} to the matrix weighted setting. First, we state it in the Roumieu case:
\beqs
\label{37LR}
\forall\;\lambda>0\ \exists\;\kappa\geq\lambda, A\geq1\ \forall\alpha,\beta\in\N_0^d:\ \ M^{(\lambda)}_\alpha
M^{(\lambda)}_\beta\leq A^{|\alpha+\beta|}M^{(\kappa)}_{\alpha+\beta};
\eeqs
and in the Beurling case:
\beqs
\label{37LB}
\forall\;\lambda>0\ \exists\;0<\kappa\leq\lambda, A\geq1\ \forall\alpha,\beta\in\N_0^d:\ \  M^{(\kappa)}_\alpha
M^{(\kappa)}_\beta\leq A^{|\alpha+\beta|}M^{(\lambda)}_{\alpha+\beta}.
\eeqs

The extensions of condition \eqref{M2'aniso} ({\itshape mixed derivation closedness properties}) for a weight matrix $\M$ in the Roumieu and Beurling cases read as follows:
\beqs
\label{M2'R}
&&\forall\;\lambda>0\ \exists\;\kappa\geq\lambda,A\geq1\  \forall\alpha\in\N_0^d,1\leq j\leq d:\ \ 
M^{(\lambda)}_{\alpha+e_j}\leq A^{|\alpha|+1}M^{(\kappa)}_\alpha,\\
\label{M2'B}
&&\forall\;\lambda>0\ \exists\;0<\kappa\leq\lambda,A\geq1\ \forall\alpha\in\N_0^d,1\leq j\leq d:\ \  
M^{(\kappa)}_{\alpha+e_j}\leq A^{|\alpha|+1}M^{(\lambda)}_\alpha.
\eeqs

The following conditions generalize \eqref{M2aniso} to the weight matrix setting:
\beqs
\label{M2R}
&&\forall\;\lambda>0\ \exists\;\kappa\geq\lambda,A\geq1\ \forall\alpha,\beta\in\N_0^d:\ \ 
M^{(\lambda)}_{\alpha+\beta}\leq A^{|\alpha+\beta|}M^{(\kappa)}_\alpha M^{(\kappa)}_\beta,\\
\label{M2B}
&&\forall\;\lambda>0\ \exists\;0<\kappa\leq\lambda,A\geq1\ \forall\alpha,\beta\in\N_0^d:\ \ 
M^{(\kappa)}_{\alpha+\beta}\leq A^{|\alpha+\beta|}M^{(\lambda)}_\alpha M^{(\lambda)}_\beta.
\eeqs

It is immediate that for any given matrix $\M$ satisfying \eqref{M2R} and \eqref{37LR} we can replace in the definition of $\Sch_{\{\M\}}$ the seminorm $\|\cdot\|_{\infty,\bfM^{(\lambda)},h}$ by 
$$\sup_{\alpha,\beta\in\N_0^d}
\frac{\|x^\alpha\partial^\beta f\|_\infty}{h^{|\alpha+\beta|}M^{(\lambda)}_{\alpha}M^{(\lambda)}_{\beta}}.$$ We have an analogous statement for the class $\Sch_{(\M)}$ under \eqref{M2B} and \eqref{37LB}. When we define the spaces $\Sch_{\{\M\}}$ or $\Sch_{(\M)}$ with the weighted $L^2$ norms treated below  in \eqref{l2norms}, the similar property holds. 

%

\begin{Lemma}
\label{lemma21G}
Let $\M$ be a weight matrix as defined in \eqref{defcalM}.

If \eqref{M2'R} holds, then
\begin{equation}
\label{22LR}
\begin{split}
&\forall\;\lambda>0\ \exists\;\kappa\geq\lambda,B_1,B_2\geq1\ \forall t\in\R^d:\\
&(1+|t|)^{2(d+1)}\exp \omega_{\bfM^{(\kappa)}}(t)\leq B_1
\exp \omega_{\bfM^{(\lambda)}}(B_2t).
\end{split}
\end{equation}

If \eqref{M2'B} holds, then
\begin{equation}
\label{22LB}
\begin{split}
&\forall\;\lambda>0\ \exists\;0<\kappa\leq\lambda,B_1,B_2\geq1\ \forall t\in\R^d:\\
&(1+|t|)^{2(d+1)}\exp \omega_{\bfM^{(\lambda)}}(t)\leq B_1
\exp \omega_{\bfM^{(\kappa)}}(B_2t).
\end{split}
\end{equation}
\end{Lemma}

\begin{proof}
First, we consider the Roumieu case. By $2(d+1)$ iterated applications of
\eqref{M2'R} we find $\kappa_{2d+2}\geq\kappa_{2d+1}\geq\ldots\geq\kappa_1\geq\lambda>0$ and $A_1,\ldots,A_{2d+2}\geq1$
such that, for all $\alpha\in\N_0^d$ and $1\leq j\leq d$,
\beqs
\nonumber
M^{(\lambda)}_{\alpha+2(d+1)e_j}\leq&&A_1^{|\alpha|+2d+2}
M^{(\kappa_1)}_{\alpha+(2d+1)e_j}\\
\nonumber
\leq&&
A_1^{|\alpha|+2d+2}A_2^{|\alpha|+2d+1}
M^{(\kappa_2)}_{\alpha+2de_j}\\
\nonumber
\le&&\cdots\leq
A_1^{|\alpha|+2d+2}A_2^{|\alpha|+2d+1}\cdots A_{2d+2}^{|\alpha|+1}
M^{(\kappa_{2d+2})}_{\alpha}\\
\label{1C}
\leq&&A^{|\alpha|+2d+2}M^{(\kappa)}_{\alpha}
\eeqs
for $A:=(\max\{A_1,\ldots,A_{2d+2}\})^{2d+2}$ and $\kappa:=\kappa_{2d+2}$.

Now, setting $|t|_{\infty}:=\max_{1\leq j\leq d}|t_j|$, we have, for $|t|_\infty\geq1$,
\beqsn
(1+|t|)^{2(d+1)}=&&\sum_{j=0}^{2(d+1)}\binom{2d+2}{j}|t|^j
\leq\sum_{j=0}^{2(d+1)}\binom{2d+2}{j}(\sqrt{d}|t|_{\infty})^j\\
\leq&& d^{d+1}|t|^{2(d+1)}_\infty\sum_{j=0}^{2(d+1)}\binom{2d+2}{j}
=(4d)^{d+1}|t|^{2(d+1)}_\infty,
\eeqsn
since $|t|=\sqrt{t_1^2+\ldots+t_d^2}\leq\sqrt{d}|t|_\infty$.
Therefore, by the definition of the associated weight function, choosing $\kappa\geq\lambda>0$ as
in \eqref{1C}, we have, assuming $|t|_\infty=t_j$ for some $1\leq j\leq d$ and $|t|_\infty\geq1$:
\beqsn
(1+|t|)^{2(d+1)}\exp\omega_{\bfM^{(\kappa)}}(t)\leq&&
(4d)^{d+1}|t_j|^{2(d+1)}\sup_{\alpha\in\N_0^d}\frac{|t^\alpha|}{M^{(\kappa)}_\alpha}\\
\leq&&(4d)^{d+1}\sup_{\alpha\in\N_0^d}
\frac{|(At)^{\alpha+2(d+1)e_j}|}{M^{(\lambda)}_{\alpha+2(d+1)e_j}}\\
\leq&&(4d)^{d+1}\sup_{\beta\in\N_0^d}
\frac{|(At)^{\beta}|}{M^{(\lambda)}_{\beta}}
=(4d)^{d+1}\exp\omega_{\bfM^{(\lambda)}}(At).
\eeqsn

On the other hand, if $t\in\R^d$ with $|t|_\infty\leq 1$, then $|t|\leq\sqrt{d}$ and hence, for $\kappa$
as in \eqref{1C},
\beqsn
(1+|t|)^{2(d+1)}\exp\omega_{\bfM^{(\kappa)}}(t)\leq C_\lambda
\leq C_\lambda\exp\omega_{\bfM^{(\lambda)}}(At),
\eeqsn
with $C_\lambda$ depending on $\lambda$ since $\kappa$ depends on $\lambda$.

We have thus proved \eqref{22LR} with $B_1=\max\{(4d)^{d+1},C_\lambda\}$
and $B_2=A$.

In the Beurling case, by $2(d+1)$ iterated applications of \eqref{M2'B}, we find
$0<\kappa_{2d+2}\leq\kappa_{2d+1}\leq\ldots\leq\kappa_1\leq\lambda$ and $A_1,\ldots,A_{2d+2}\geq1$ such that
\beqs\label{iter-B}
M^{(\lambda)}_\alpha\geq&&A_1^{-|\alpha|-1}M^{(\kappa_1)}_{\alpha+e_j}
\geq A_1^{-|\alpha|-1}A_2^{-|\alpha|-2}M^{(\kappa_2)}_{\alpha+2e_j}\\ \nonumber
\ldots\geq&&A_1^{-|\alpha|-1}A_2^{-|\alpha|-2}\cdots A_{2d+2}^{-|\alpha|-2d-2}
M^{(\kappa_{2d+2})}_{\alpha+2(d+1)e_j}
\geq A^{-|\alpha|-2d-2}
M^{(\kappa)}_{\alpha+2(d+1)e_j},
\eeqs
for $A:=(\max\{A_1,\ldots,A_{2d+2}\})^{2d+2}$ and $\kappa:=\kappa_{2d+2}$.
Then we proceed as in the Roumieu case and prove that
\beqsn
(1+|t|)^{2(d+1)}\exp\omega_{\bfM^{(\lambda)}}(t)
\leq B_1'\exp\omega_{\bfM^{(\kappa)}}(At),
\eeqsn
for $B_1':=\max\{(4d)^{d+1},\max_{|t|\leq\sqrt{d}}
(1+|t|)^{2(d+1)}\exp\omega_{\bfM^{(\lambda)}}(t)\}$.

The proof is complete.
\end{proof}

\begin{Lemma}
\label{lemma58G}
Let $\M$ be a weight matrix that satisfies \eqref{M2'B}. Then
\begin{equation}
\label{5BJO-Rodino}
\begin{split}
&\forall\;\lambda>0, N\in\N\ \exists\;0<\kappa\leq\lambda, A,B\geq1\ \forall t\in\R^d\setminus\{0\}:\\
&\omega_{\bfM^{(\lambda)}}(t)+N\log|t|\leq\omega_{\bfM^{(\kappa)}}(At)+B.
\end{split}
\end{equation}
Let $\M$ be a weight matrix that satisfies \eqref{M2'R}. Then
\begin{equation}
\label{5BJO-Rodino1}
\begin{split}
&\forall\;\lambda>0, N\in\N\ \exists\;\kappa\geq\lambda, A,B\geq1\ \forall t\in\R^d\setminus\{0\}: \\
&\omega_{\bfM^{(\kappa)}}(t)+N\log|t|\leq\omega_{\bfM^{(\lambda)}}(At)+B.
\end{split}
\end{equation}
\end{Lemma}

\begin{proof}
If $t\in\R^d\setminus\{0\}$, then by the definition of the associated weight function, for $1\leq j\leq d$ such that $|t|_\infty=t_j$, 
\beqs
\nonumber
|t|^N\exp\omega_{\bfM^{(\lambda)}}(t)\leq&&(\sqrt{d}|t|_\infty)^N\exp
\omega_{\bfM^{(\lambda)}}(t)
=d^{N/2}|t_j|^N\exp\omega_{\bfM^{(\lambda)}}(t)\\
\label{prel1}
=&&d^{N/2}|t^{Ne_j}|\sup_{\alpha\in\N_{0,t}^d}\frac{|t^\alpha|}{M^{(\lambda)}_\alpha}
=d^{N/2}
\sup_{\alpha\in\N_{0,t}^d}\frac{|t^{\alpha+Ne_j}|}{M^{(\lambda)}_\alpha},
\eeqs
where $\N^d_{0,t}$ is defined by \eqref{4}. This estimate is valid for any given index $\lambda>0$. 

In the Beurling case, by $N$ iterated applications of \eqref{M2'B} we find $\kappa_N\leq\kappa_{N-1}\leq\ldots\leq\kappa_1\leq\lambda$
and $A_1,\ldots,A_N\geq1$ such that, for $A:=(\max\{A_1,\ldots,A_N\})^N$ and
$\kappa:=\kappa_N$, we have, proceeding as in \eqref{iter-B}, $M^{(\kappa)}_{\alpha+Ne_j}\le A^{|\alpha|+N} M^{(\lambda)}_\alpha.$
Therefore 
\beqsn
|t|^N\exp\omega_{\bfM^{(\lambda)}}(t)
\leq d^{N/2}
\sup_{\alpha\in\N_{0,t}^d}
\frac{|(At)^{\alpha+Ne_j}|}{M^{(\kappa)}_{\alpha+Ne_j}}
\leq d^{N/2}\exp\omega_{\bfM^{(\kappa)}}(At),
\eeqsn
and we conclude that \eqref{5BJO-Rodino} is satisfied for $B:=\max\{\frac N2\log d,1\}$.

In the Roumieu case we make $N$ iterated applications of \eqref{M2'R} and we find indices $\kappa:=\kappa_N\geq\kappa_{N-1}\geq\ldots\geq\kappa_1\geq\lambda$ and $A_1,\ldots,A_N\geq1$ such that, for $A:=(\max\{A_1,\ldots,A_N\})^N$ and $\kappa=\kappa_N$, as in \eqref{1C} we have that
$M^{(\lambda)}_{\alpha+Ne_j}\leq A^{|\alpha|+N}M^{(\kappa)}_\alpha$ and hence
from \eqref{prel1}:
\beqsn
|t|^N\exp\omega_{\bfM^{(\kappa)}}(t)\leq d^{N/2}
\sup_{\alpha\in\N_{0,t}^d}
\frac{|t^{\alpha+Ne_j}|}{M^{(\kappa)}_{\alpha}}
\leq d^{N/2}
\sup_{\alpha\in\N_{0,t}^d}
\frac{|(At)^{\alpha+Ne_j}|}{M^{(\lambda)}_{\alpha+Ne_j}}
\le d^{N/2}\exp\omega_{\bfM^{(\lambda)}}(At),
\eeqsn
so that \eqref{5BJO-Rodino1} is satisfied with $B=\max\{\frac N2\log d,1\}$.
\end{proof}

Now, we  consider the different system of seminorms
\beqs\label{l2norms}
\|f\|_{2,{\bf M}^{(\lambda)},h}:=\sup_{\alpha,\beta\in\N_0^d}
\frac{\|x^\alpha\partial^\beta f\|_2}{h^{|\alpha+\beta|}M^{(\lambda)}_{\alpha+\beta}},
\qquad\lambda,h>0,
\eeqs
on $\Sch_{(\M)}$ and $\Sch_{\{\M\}}$, where $\|\cdot\|_2$ is the $L^2$ norm. Under suitable conditions on the weight matrix $\M$, it turns out to be equivalent to the previous one given by sup norms, as we prove in the following

\begin{Prop}
\label{lemma25G}
Let $\M$ be a weight matrix as defined in \eqref{defcalM} that satisfies \eqref{12L2B} and \eqref{M2'B} (\eqref{12L2R} and
\eqref{M2'R}). Then the system of seminorms $\|\cdot\|_{\infty,{\bf M}^{(\lambda)},h}$
in $\Sch_{(\M)}$ ($\Sch_{\{\M\}}$) is equivalent to the system of seminorms
$\|\cdot\|_{2,{\bf M}^{(\lambda)},h}$. More precisely, in the Beurling case we have the following two conditions for every $f\in C^\infty(\R^d)$
\beqs
\label{dis1}
\exists\;C_1>0\ \forall\, \lambda,h>0\ \exists\;\kappa>0, \tilde h=\tilde{h}_{\lambda,h}>0:\quad
\|f\|_{2,{\bf M}^{(\lambda)},h}\leq C_1\|f\|_{\infty,{\bf M}^{(\kappa)},\tilde h},
\eeqs
\beqs
\label{dis2}
\forall\,\lambda,h>0\ \exists\;\widetilde{\kappa}>0,  C_{\lambda,h}>0,\tilde h=\tilde h_{\lambda,h}>0:
\quad\|f\|_{\infty,{\bf M}^{(\lambda)},h}\leq C_{\lambda,h}\|f\|_{2,{\bf M}^{(\widetilde{\kappa})},\tilde h}\,;
\eeqs
in the Roumieu case we have the following two conditions, for every $f\in C^\infty(\R^d)$,
\beqs
\label{dis11}
\forall\;\lambda,h>0\,\exists\;C_{\lambda,h}>0,\,\exists\;\kappa\geq\lambda,\tilde h>0:\quad
\|f\|_{2,{\bf M}^{(\kappa)},\tilde h}\leq C_{\lambda,h}\|f\|_{\infty,{\bf M}^{(\lambda)},h},
\eeqs
\beqs
\label{dis4}
\forall\,\lambda,h>0\ \exists\, C_{\lambda,h}>0, \widetilde{\kappa}>0,\tilde h>0:
\quad\|f\|_{\infty,{\bf M}^{(\widetilde{\kappa})},\tilde h}\leq C_{\lambda,h}\|f\|_{2,{\bf M}^{(\lambda)},h}\,.
\eeqs
\end{Prop}

\begin{proof}
Let $f\in C^\infty(\R^d)$. Then, for $C_1=(\int_{\R^d}\frac{1}{(1+|x|^2)^{d+1}}dx)^{1/2}$, we have 
\beqsn
\|x^\alpha\partial^\beta f\|_2 
\le C_1\|(1+|x|^2)^{\frac{d+1}{2}}x^\alpha\partial^\beta f(x)\|_\infty.
\eeqsn
If $|x|_\infty\leq1$, then
\beqsn
(1+|x|^2)^{\frac{d+1}{2}}\leq(1+d|x|_\infty^2)^{\frac{d+1}{2}}
\leq(1+d)^{\frac{d+1}{2}}.
\eeqsn
On the other hand, if $|x|_\infty\geq1$ then
\beqsn
(1+|x|^2)^{\frac{d+1}{2}}\leq(|x|_\infty^2+|x|^2)^{\frac{d+1}{2}}
\leq(|x|_\infty^2+d|x|_\infty^2)^{\frac{d+1}{2}}
\leq(d+1)^{\frac{d+1}{2}}|x|_\infty^{d+1}.
\eeqsn
Therefore, for any fixed $x\in\R^d$, being $|x|_\infty=|x_j|$ for some $1\leq j\leq d$, we have
\beqsn
|(1+|x|^2)^{\frac{d+1}{2}}x^\alpha|\leq(d+1)^{\frac{d+1}{2}}
\max\{|x^{\alpha+(d+1)e_j}|,|x^\alpha|\}
\eeqsn
and hence
\begin{equation}
\label{5C}
\begin{split}
\|x^\alpha\partial^\beta f\|_2\leq C_1(d+1)^{\frac{d+1}{2}}
\max\{&\|x^{\alpha+(d+1)e_1}\partial^\beta f\|_\infty,
\|x^{\alpha+(d+1)e_2}\partial^\beta f\|_\infty,\\
&\ldots,\|x^{\alpha+(d+1)e_d}\partial^\beta f\|_\infty,\|x^\alpha\partial^\beta f\|_\infty\}.
\end{split}
\end{equation}

Now, we consider separately the Beurling and Roumieu cases.
In the Beurling case, for every $\lambda,h>0$ we first estimate $\|x^{\alpha+(d+1)e_j}
\partial^\beta f\|_{2,{\bf M}^{(\lambda)},h}$ in order to use \eqref{5C}.
By $(d+1)$ iterated applications of \eqref{M2'B} there exist
$0<\kappa:=\kappa_{d+1}\leq\kappa_d\leq\ldots\leq\kappa_1\leq\lambda$ and
$A_1,\ldots,A_{d+1}\geq1$ ($A_j$ depending on $\lambda$) such that, proceeding as in \eqref{iter-B}, we obtain $M^{(\kappa)}_{\alpha+\beta+(d+1)e_j}\le A_{\lambda}^{|\alpha+\beta|+d+1} M^{(\lambda)}_{\alpha+\beta}$ for $A_\lambda=(\max\{A_1,\ldots,A_{d+1}\})^{d+1}\geq1$. 
Hence, we deduce 
\beqsn
\frac{\|x^{\alpha+(d+1)e_j}\partial^\beta f\|_\infty}{h^{|\alpha+\beta|}M^{(\lambda)}_{\alpha+\beta}}\le
\frac{\|x^{\alpha+(d+1)e_j}\partial^\beta f\|_\infty}{h^{|\alpha+\beta|+d+1}
M^{(\kappa)}_{\alpha+\beta+(d+1)e_j}}\cdot h^{d+1}A_{\lambda}^{|\alpha+\beta|+d+1}.
\eeqsn

Therefore, from \eqref{5C} and the fact that ${\bf M}^{(\kappa)}\leq {\bf M}^{(\lambda)}$, we have for every $\lambda,h>0$,
\begin{equation}
\label{5CC}
\begin{split}
\frac{\|x^\alpha\partial^\beta f\|_2}{h^{|\alpha+\beta|}M^{(\lambda)}_{\alpha+\beta}}
\leq& C_1(d+1)^{\frac{d+1}{2}}
\max\bigg\{\frac{\|x^{\alpha+(d+1)e_1}\partial^\beta f\|_\infty}{h^{|\alpha+\beta|+d+1}M^{(\kappa)}_{\alpha+\beta+(d+1)e_1}}
h^{d+1}A_{\lambda}^{|\alpha+\beta|+d+1},\\
&\ldots,\frac{\|x^{\alpha+(d+1)e_d}\partial^\beta f\|_\infty}{h^{|\alpha+\beta|+d+1}M^{(\kappa)}_{\alpha+\beta+(d+1)e_d}}
h^{d+1}A_{\lambda}^{|\alpha+\beta|+d+1},
\frac{\|x^\alpha\partial^\beta f\|_\infty}{h^{|\alpha+\beta|}M^{(\kappa)}_{\alpha+\beta}}
\bigg\}.
\end{split}
\end{equation}
If $h\geq1$ then $h^{d+1}A_\lambda^{|\alpha+\beta|+d+1}\leq
(hA_\lambda)^{|\alpha+\beta|+d+1}$.
If $0<h<1$ then $h^{d+1}A_\lambda^{|\alpha+\beta|+d+1}\leq A_\lambda^{|\alpha+\beta|+d+1}$. Hence, for 
\beqsn
\tilde h:=\begin{cases}
\min\left\{\frac{1}{A_\lambda},h\right\}=\frac{1}{A_\lambda},&\mbox{if}\ h\geq 1,\cr
\min\left\{\frac{h}{A_\lambda},h\right\}=\frac{h}{A_\lambda},&\mbox{if}\ 0<h<1,
\end{cases}
\eeqsn
we obtain
\beqsn
\|f\|_{2,{\bf M}^{(\lambda)},h}\leq C_1(d+1)^{\frac{d+1}{2}}\|f\|_{\infty,{\bf M}^{(\kappa)},\tilde h}.
\eeqsn
This shows \eqref{dis1}. 

Now, since $\delta!\leq \delta_1^{\delta_1}\dots \delta_d^{\delta_d}=\delta^\delta$, we have
$
\frac{\alpha!}{(\alpha-\delta)!}\le \binom\alpha\delta\delta!\le 2^{|\alpha|}\delta^{\delta}.
$
So it follows by Leibnitz's rule and \cite[formula (2.3)]{L} that, for some $C_2>0$, 
\beqs
\nonumber
\|x^\alpha\partial^\beta f\|_\infty \leq &&C_2\sup_{|\gamma|_\infty\leq 2d+2}
\|\partial^\gamma(x^\alpha\partial^\beta f)\|_2\\ \nonumber
\leq&& C_2\sup_{|\gamma|_\infty\leq 2d+2}
\sum_{\delta\leq\gamma}\binom\gamma\delta
\|(\partial^\delta x^\alpha)\partial^{\beta+\gamma-\delta}f\|_2\\
%
\label{7C}
\leq&&C_2
\sup_{|\gamma|_\infty\leq 2d+2}
\sum_{\afrac{\delta\leq\gamma}{\delta\leq\alpha}}\binom\gamma\delta
2^{|\alpha|}\delta^{\delta}
\|x^{\alpha-\delta}\partial^{\beta+\gamma-\delta}f\|_2\,.
\eeqs
On the other hand, by $|\gamma|$ iterated applications of \eqref{M2'B}, there exist
$0<\kappa:=\kappa_{|\gamma|}\leq\kappa_{|\gamma|-1}\leq\ldots\leq\kappa_1\leq\lambda$
and $A_1,\ldots,A_{|\gamma|}\geq1$ such that, for $A_\lambda:=\max\limits_{\vert\gamma\vert_\infty\leq 2d+2}(\max\{A_1,\ldots,A_{|\gamma|}\})^{|\gamma|}$, we have $M^{(\kappa)}_{\alpha+\beta+\gamma}\le A_\lambda^{|\alpha+\beta+\gamma|} M^{(\lambda)}_{\alpha+\beta}$.
By \eqref{12L2B}, there exist $0<\widetilde{\kappa}\leq\kappa$ and $H>0$
such that for all $C>0$ there is $B>0$ so that 
\beqsn
\frac{\|x^\alpha\partial^\beta f\|_\infty}{h^{|\alpha+\beta|}M^{(\lambda)}_{\alpha+\beta}}\leq&&
C_2
\sup_{|\gamma|_\infty\leq 2d+2}
\sum_{\afrac{\delta\leq\gamma}{\delta\leq\alpha}}\binom\gamma\delta
\frac{\|x^{\alpha-\delta}\partial^{\beta+\gamma-\delta}f\|_2}{h^{|\alpha+\beta+\gamma-2\delta|}
M^{(\widetilde{\kappa})}_{\alpha+\beta+\gamma-2\delta}}\cdot h^{|\gamma-2\delta|}\\
&&\cdot 2^{|\alpha|}A_\lambda^{|\alpha+\beta+\gamma|}
BC^{|2\delta|}H^{|\alpha+\beta+\gamma|}.
\eeqsn
Observe that $\widetilde{\kappa}$ may depend on $\gamma$. From \eqref{defcalM} we can consider in the previous estimates, instead of $\widetilde{\kappa}$, the minimum of all these $\widetilde{\kappa}$ for $\vert\gamma\vert_\infty\leq 2d+2$, so that we can finally choose $\widetilde{\kappa}$ independent of $\gamma$. Since $|\gamma|\leq d|\gamma|_\infty\leq2d(d+1)$ we have
\beqsn
\frac{\|x^\alpha\partial^\beta f\|_\infty}{h^{|\alpha+\beta|}M^{(\lambda)}_{\alpha+\beta}} 
\leq&&C_2B(2CHA_\lambda)^{4d(d+1)} \sup_{\vert\gamma\vert_\infty\leq 2d+2}
\sum_{\afrac{\delta\leq\gamma}{\delta\leq\alpha}}\binom\gamma\delta
\frac{\|x^{\alpha-\delta}\partial^{\beta+\gamma-\delta}f\|_2}{h^{|\alpha+\beta+\gamma-2\delta|}M^{(\widetilde{\kappa})}_{\alpha+\beta+\gamma-2\delta}}\\
&&\cdot(2HA_\lambda)^{|\alpha+\beta+\gamma-2\delta|}
h^{|\gamma-2\delta|}.
\eeqsn
Now, if $h\geq1$, then $h^{|\gamma-2\delta|}\leq h^{|\alpha+\beta+\gamma-2\delta|}$. And if $0<h<1$, then $h^{|\gamma-2\delta|}\leq 1$ when $|\gamma-2\delta|\ge 0$ and $h^{|\gamma-2\delta|}\le h^{-|\gamma|}
\leq h^{-2d(d+1)}$ when $|\gamma-2\delta|<0$. For
\beqs
\label{htilde2}
\tilde h=\begin{cases}
\frac{1}{2HA_\lambda}&\mbox{if}\ h\geq1\cr
\frac{h}{2HA_\lambda}&\mbox{if}\ 0<h<1,
\end{cases}
\eeqs
taking into account that
\beqsn
\sum_{\delta\leq\gamma}\binom\gamma\delta\leq d^{|\gamma|}\leq d^{2d(d+1)},
\eeqsn
we finally have that for all $\lambda,h>0$ there exist $\widetilde{\kappa}$, $C_{\lambda,h}>0$ and
$\tilde h>0$, such
that
\beqs
\label{dis2bis}
\|f\|_{\infty,{\bf M}^{(\lambda)},h}\leq C_{\lambda,h}\|f\|_{2,{\bf M}^{(\widetilde{\kappa})},\tilde h}\,.
\eeqs
Since neither $H$ nor $A_{\lambda}$ are depending on $h$, we have  $\tilde{h}\rightarrow 0$ as $h\rightarrow 0$. This shows \eqref{dis2} and concludes the proof in the Beurling case.

Let us now consider the Roumieu case. In \eqref{5C}, for any given $\lambda$ by $(d+1)$ iterated applications
of \eqref{M2'R}, we obtain  $\kappa:=\kappa_{d+1}\geq\kappa_d\geq\ldots\geq\kappa_1\geq\lambda>0$
and $A_1,\ldots,A_{d+1}\geq1$ such that, for $A_\lambda:=(\max\{A_1,\ldots,A_{d+1}\})^{d+1}$, we have $M^{(\lambda)}_{\alpha+\beta+(d+1)e_j}\le  A_\lambda^{|\alpha+\beta|+d+1}M^{(\kappa)}_{\alpha+\beta}$.
Then from \eqref{5C} and the fact that $M^{(\kappa)}_{\alpha+\beta}\geq M^{(\lambda)}_{\alpha+\beta}$ we obtain, given a fixed $h>0$, for $\tilde{h}:=\max\{ hA_\lambda,1\}$,
\beqsn
\frac{\Vert x^\alpha\partial^\beta f\Vert_2}{M^{(\kappa)}_{\alpha+\beta}}&\leq&  C_1 (d+1)^{\frac{d+1}{2}} \max\Biggr\{ \frac{\Vert x^\alpha \partial^\beta f\Vert_\infty}{M^{(\kappa)}_{\alpha+\beta}},\ \frac{\Vert x^{\alpha+(d+1)e_j}\partial^\beta f\Vert_\infty}{M^{(\lambda)}_{\alpha+\beta+(d+1)e_j}} A_\lambda^{\vert\alpha+\beta\vert+d+1}\Biggr\} \\
&\leq& C_1 (d+1)^{\frac{d+1}{2}} \tilde{h}^{\vert\alpha+\beta\vert+d+1} \max\Biggr\{ \frac{\Vert x^\alpha \partial^\beta f\Vert_\infty}{h^{\vert\alpha+\beta\vert}M^{(\lambda)}_{\alpha+\beta}},\ \frac{\Vert x^{\alpha+(d+1)e_j}\partial^\beta f\Vert_\infty}{h^{\vert\alpha+\beta\vert+d+1}M^{(\lambda)}_{\alpha+\beta+(d+1)e_j}} \Biggr\}.
\eeqsn
Hence, dividing by $\tilde{h}^{\vert\alpha+\beta\vert}$,
\beqsn
\frac{\Vert x^\alpha\partial^\beta f\Vert_2}{\tilde{h}^{\vert\alpha+\beta\vert}M^{(\kappa)}_{\alpha+\beta}}\leq C_1 (d+1)^{\frac{d+1}{2}} \tilde{h}^{d+1} \Vert f\Vert_{\infty,{\bf M}^{(\lambda)},h};
\eeqsn
then \eqref{dis11} is proved, with $C_{\lambda,h}=C_1 (d+1)^{\frac{d+1}{2}} \tilde{h}^{d+1}$ (observe that $\tilde{h}$ depends on $h$ and $\lambda$).

Now,  given any  $\lambda>0$ consider $\kappa\geq\lambda>0$ and $B,C,H>0$
as in \eqref{12L2R}. Then, by $|\gamma|$ iterated applications of \eqref{M2'R}, there exist
$\widetilde{\kappa}:=\kappa_{|\gamma|}\geq\ldots\geq\kappa_1\geq\kappa\geq\lambda$ and
$A_1,\ldots,A_{|\gamma|}\geq1$ such that, for $A_\lambda:=(\max\{A_1,\ldots,A_{|\gamma|}\})^{|\gamma|}$, $M^{(\kappa)}_{\alpha+\beta+\gamma}\leq A_\lambda^{|\alpha+\beta+\gamma|}M^{(\tilde\kappa)}_{\alpha+\beta}.$ 
So, from \eqref{7C} with $h=1$ and $\tilde\kappa$ instead of $\lambda$,
applying \eqref{12L2R} and proceeding as before, we get
\beqs
\label{19dec2}
\frac{\|x^\alpha\partial^\beta f\|_\infty}{M^{(\tilde\kappa)}_{\alpha+\beta}}\leq 
C_2 BC^{4d(d+1)}
\sup_{|\gamma|_\infty\leq 2d+2}
\sum_{\afrac{\delta\leq\gamma}{\delta\leq\alpha}}\binom\gamma\delta
(2A_\lambda H)^{|\alpha+\beta+\gamma|}
\frac{\|x^{\alpha-\delta}\partial^{\beta+\gamma-\delta}f\|_2}{M^{(\lambda)}_{\alpha+\beta+\gamma-2\delta}}\,.
\eeqs
Since for every $h>0$ and $\alpha,\beta,\gamma,\delta$ as above
\beqsn
\frac{\|x^{\alpha-\delta}\partial^{\beta+\gamma-\delta}f\|_2}{h^{\vert\alpha+\beta+\gamma-2\delta\vert} M^{(\lambda)}_{\alpha+\beta+\gamma-2\delta}}\leq \Vert f\Vert_{2,{\bf M}^{(\lambda)},h}\,,
\eeqsn
dividing \eqref{19dec2} by $(2A_\lambda Hh)^{\vert\alpha+\beta\vert}$ we obtain
\beqsn
\frac{\Vert x^\alpha\partial^\beta f\Vert_{\infty}}{(2A_\lambda Hh)^{\vert\alpha+\beta\vert} M^{(\widetilde{\kappa})}_{\alpha+\beta}}\leq \Vert f\Vert_{2,{\bf M}^{(\lambda)},h}\, C_2 BC^{4d(d+1)}
\sup_{|\gamma|_\infty\leq 2d+2}
\sum_{\delta\leq\gamma}\binom\gamma\delta
(2A_\lambda Hh)^{|\gamma|} h^{-\vert 2\delta\vert}.
\eeqsn
Taking the $\sup$ on $\alpha$ and $\beta$ in the left-hand side, we then get \eqref{dis4} with $\tilde{h}=2A_\lambda Hh$ and
\beqsn
C_{\lambda,h}=C_2 BC^{4d(d+1)}
\sup_{|\gamma|_\infty\leq 2d+2}
\sum_{\delta\leq\gamma}\binom\gamma\delta
(2A_\lambda Hh)^{|\gamma|} h^{-\vert 2\delta\vert}.
\eeqsn
\end{proof}

We observe  that in \eqref{dis1} the constant $C_1$ is fixed (it depends only on the dimension d), and moreover we only need \eqref{M2'B} to prove it.
On the other hand, to obtain \eqref{dis2} we consider \eqref{M2'B} and \eqref{12L2B}.
In the Roumieu case we just need \eqref{M2'R} to prove \eqref{dis11}, while
for the proof of \eqref{dis4} we use \eqref{12L2R} to choose $\kappa\geq\lambda$ 
and then \eqref{M2'R} to get
$\tilde \kappa\geq\kappa$.

\section{Hermite functions: properties in the matrix setting}
\label{sec3}

We recall the definition of the {\em Hermite functions} $H_\gamma$ for $\gamma\in\N_0^d$:
\beqsn
H_\gamma(x):=(2^{|\gamma|}\gamma!\pi^{d/2})^{-1/2}h_\gamma(x)
\exp\left(-\sum_{j=1}^d\frac{x_j^2}{2}\right),\qquad x\in\R^d,
\eeqsn
where $h_\gamma$ are the Hermite polynomials
\beqsn
h_\gamma(x):=(-1)^{|\gamma|}\exp\left(\sum_{j=1}^dx_j^2\right)\cdot\partial^\gamma
\exp\left(-\sum_{j=1}^dx_j^2\right),\qquad x\in\R^d.
\eeqsn
As in \cite{L} we consider, for $f\in C^\infty(\R^d)$, the operators
\beqsn
&&A_{\pm,i}(f):=\mp\partial_{x_i}f+x_i f,\qquad 1\leq i\leq d,\\
&&A_{\pm}^\alpha(f):=\prod_{i=1}^dA_{\pm,i}^{\alpha_i}(f),\qquad\alpha\in\N_0^d,
\eeqsn
with $A_{\pm,i}^0:=\id$.

By \cite[Example 29.5(2)]{MV} setting $H_\beta=0$ if $\beta_j=-1$ for some
$1\leq j\leq d$, we have
\beqsn
A_{-,j}(H_\gamma)=\sqrt{2\gamma_j}H_{\gamma-e_j},\quad\gamma\in\N_0^d.
\eeqsn
It follows that
\beqs
\label{15C}
\ A_{-}^\alpha(H_{\gamma+\alpha})=\prod_{1\leq j\leq d}A_{-,j}^{\alpha_j}(H_{\gamma+\alpha})
=\prod_{1\leq j\leq d}(\sqrt{2\gamma_j})^{\alpha_j}H_\gamma
=\sqrt{2^{|\alpha|}\gamma^\alpha}H_\gamma\,,\quad \alpha,\gamma\in\N_0^d.
\eeqs

We also recall the following two lemmas from \cite{L}:
\begin{Lemma}
\label{lemma31aL}
Let $f\in C^\infty(\R^d)$. Then, for all $\gamma\in\N_0^d$ and $x\in\R^d$,
\beqsn
(A_+^\gamma f)(x)=\sum_{\alpha+\beta\leq\gamma}C_{\alpha,\beta}(\gamma)
x^\alpha\partial^\beta f(x),
\eeqsn
for some coefficients $C_{\alpha,\beta}(\gamma)$ satisfying
\beqsn
|C_{\alpha,\beta}(\gamma)|\leq 3^{|\gamma|}\left(\frac{\gamma!}{(\alpha+\beta)!}\right)^{1/2}
\!\!\!,\qquad
\alpha,\beta,\gamma\in\N_0^d.
\eeqsn
\end{Lemma}

\begin{Lemma}
\label{lemma32aL}
For all $\alpha,\beta,\gamma\in\N_0^d$
\beqsn
\|x^\alpha\partial^\beta H_\gamma\|_2\leq 2^{\frac{|\alpha+\beta|}{2}}
\left(\frac{(\alpha+\beta+\gamma)!}{\gamma!}\right)^{1/2}.
\eeqsn
\end{Lemma}

We can generalize Lemma 3.1(b) of \cite{L} in the following way:
\begin{Lemma}
\label{lemma33G}
Let $\bfM=(M_\alpha)_{\alpha\in\N_0^d}$ and $\bfN=(N_\alpha)_{\alpha\in\N_0^d}$ be
two sequences satisfying \eqref{12L1R} for some $C,B,H>0$.
Assume that $f\in C^\infty(\R^d)$ satisfies, for some $C_1>0$ and for the same constant $C$ as in
\eqref{12L1R},
\beqs
\label{31G}
\|f\|_{2,\bfM,C}=\sup_{\alpha,\beta\in\N_0^d}\frac{\|x^\alpha\partial^\beta f\|_2}{C^{|\alpha+\beta|}M_{\alpha+\beta}}\leq C_1.
\eeqs
Then
\beqsn
\|A_+^\gamma f\|_2\leq C_1Be^{d/2}(9\sqrt{2}HC)^{|\gamma|}N_\gamma,
\quad\gamma\in\N_0^d.
\eeqsn
\end{Lemma}

\begin{proof}
By Stirling's inequality $e\left(\frac{n}{e}\right)^n\leq n!\leq en\left(\frac{n}{e}\right)^n$ for any $n\in \N$. Hence, by Lemma~\ref{lemma31aL} and the  assumption \eqref{31G}, we have 
\beqsn
\|A_+^\gamma f\|_2\leq&&\sum_{\alpha+\beta\leq\gamma}|C_{\alpha,\beta}(\gamma)|\cdot\|x^\alpha
\partial^\beta f\|_2\\
\leq&& C_13^{|\gamma|}\sum_{\alpha+\beta\leq\gamma}\binom{\gamma}{\alpha+\beta}^{1/2}
(\gamma-\alpha-\beta)!^{1/2}C^{|\alpha+\beta|}M_{\alpha+\beta}\\
\leq&&C_13^{|\gamma|}\sum_{\alpha+\beta\leq\gamma}\binom{\gamma}{\alpha+\beta}^{1/2}
e^{d/2}\left(\prod_{j=1}^d(\gamma_j-\alpha_j-\beta_j)^{1/2}\right) \frac{(\gamma-\alpha-\beta)^{\frac{\gamma-\alpha-\beta}{2}}}{\exp\left\{\frac{\vert\gamma-\alpha-\beta\vert}{2}\right\}} C^{|\alpha+\beta|}
M_{\alpha+\beta}\\
\leq&&C_13^{|\gamma|}2^{|\gamma|/2}e^{d/2}
\sum_{\alpha+\beta\leq\gamma}\binom{\gamma}{\alpha+\beta}
(\gamma-\alpha-\beta)^{\frac{\gamma-\alpha-\beta}{2}}C^{|\alpha+\beta|}
M_{\alpha+\beta}.
\eeqsn

Applying now \eqref{12L1R} and $\sum_{\alpha+\beta\leq\gamma}
\binom{\gamma}{\alpha+\beta}\leq 3^{|\gamma|}$ (by
\cite[pg 274]{L}), we get
\beqsn
\|A_+^{\gamma} f\|_2\leq&&C_1 e^{d/2}(3\sqrt{2})^{|\gamma|}
\sum_{\alpha+\beta\leq\gamma}\binom{\gamma}{\alpha+\beta}
BC^{|\gamma-\alpha-\beta|}H^{|\gamma|}N_\gamma C^{|\alpha+\beta|}\\
\leq&&C_1B e^{d/2}(9\sqrt{2})^{|\gamma|}(CH)^{|\gamma|}N_\gamma.
\eeqsn
\end{proof}

As a corollary, we immediately have the following
\begin{Lemma}
\label{cor34G}
Let $\M$ be a weight matrix satisfying \eqref{12L2R} and assume that $f\in C^\infty(\R^d)$
satisfies, for some $\lambda,C_1>0$
\beqs
\label{31'G}
\|f\|_{2,\bfM^{(\lambda)},C}\leq C_1
\eeqs
for the constant $C$ of \eqref{12L2R}. Then
\beqsn
\|A_+^\gamma f\|_2\leq C_1B e^{d/2}(9\sqrt{2}HC)^{|\gamma|}M^{(\kappa)}_\gamma,
\qquad\forall\gamma\in\N_0^d,
\eeqsn
with $\kappa,B,H,C$ as in \eqref{12L2R}.

If $\M$ satisfies \eqref{12L2B} and if, for some $\lambda>0$, $f\in C^\infty(\R^d)$
satisfies
\beqsn
\|f\|_{2,\bfM^{(\kappa)},C}\leq C_1
\eeqsn
for the constant $\kappa\leq\lambda$ of \eqref{12L2B} and for some $C,C_1>0$, then
\beqsn
\|A_+^\gamma f\|_2\leq C_1B e^{d/2}(9\sqrt{2}HC)^{|\gamma|}M^{(\lambda)}_\gamma,
\qquad\forall\gamma\in\N_0^d,
\eeqsn
where $H=H(\lambda)$ and $B=B(C,\lambda)$ are given by \eqref{12L2B}.
\end{Lemma}

The following lemma generalizes \cite[Lemma 3.2(b)]{L}.
\begin{Lemma}
\label{lemma35G}
Let $\bfM=(M_\alpha)_{\alpha\in\N_0^d}$ and $\bfN=(N_\alpha)_{\alpha\in\N_0^d}$
be two weight sequences satisfying \eqref{12L1R}. Then
\beqsn
\|H_\gamma\|_{2,\bfN,2HC}=
\sup_{\alpha,\beta\in\N_0^d}\frac{\|x^\alpha\partial^\beta H_\gamma\|_2}{(2HC)^{|\alpha+\beta|}N_{\alpha+\beta}}\leq
Be^{\omega_\bfM(\gamma^{1/2}/C)},
\qquad\forall\gamma\in\N_0^d,
\eeqsn
where $\gamma^{1/2}:=(\gamma_1^{1/2},\ldots,\gamma_d^{1/2})$ and
$B,C,H>0$ are the constants in \eqref{12L1R}.
\end{Lemma}

\begin{proof}
For $\alpha,\beta,\gamma\in\N_0^d$ we set
\beqsn
&&J:=\{j\in\N:\ 1\leq j\leq d,\alpha_j+\beta_j\leq\gamma_j\}\\
&&J^c:=\{j\in\N:\ 1\leq j\leq d,\alpha_j+\beta_j>\gamma_j\}.
\eeqsn
Then for any $\delta\in\N^d$ we denote
\beqsn
\delta_J:=\sum_{j\in J}\delta_je_j,
\quad \delta_{J^c}:=\sum_{j\in J^c}\delta_je_j,
\eeqsn
so that $\delta=\delta_J+\delta_{J^c}$. By Lemma~\ref{lemma32aL} and \eqref{12L1R}, we have
\beqs
\nonumber
\|x^\alpha\partial^\beta H_\gamma\|_2\leq&&2^{\frac{|\alpha+\beta|}{2}}
\left(\frac{(\alpha+\beta+\gamma)!}{\gamma!}\right)^{1/2}
\leq2^{\frac{|\alpha+\beta|}{2}}(\alpha+\beta+\gamma)^{\frac{\alpha+\beta}{2}}\\
\nonumber
\leq 
&&2^{|\alpha+\beta|}(\alpha_{J^c}+\beta_{J^c})^{\frac{\alpha_{J^c}+\beta_{J^c}}{2}}\gamma_J^{\frac{\alpha_J+\beta_J}{2}}\\
\label{10C}
\leq&&B(2HC)^{|\alpha+\beta|}N_{\alpha+\beta}\gamma_J^{\frac{\alpha_J+\beta_J}{2}}
\frac{1}{M_{\alpha_J+\beta_J}C^{|\alpha_J+\beta_J|}}.
\eeqs

Now, since $\alpha_J$ has the $j$-th entry equal to $\alpha_j$ for $j\in J$ and $0$ for
$j\in J^c$, 
\beqs
\label{11C}
\gamma_J^{\frac{\alpha_J+\beta_J}{2}}=\prod_{j\in J}
\gamma_j^{\frac{\alpha_j+\beta_j}{2}}=\prod_{j\in J}
\gamma_j^{\frac{\alpha_j+\beta_j}{2}}
\prod_{j\in J^c}\gamma_j^0=\gamma^{\frac{\alpha_J+\beta_J}{2}}.
\eeqs
Moreover, by Lemma~\ref{lemma2},
\beqs
\label{12C}
M_{\alpha_J+\beta_J}C^{|\alpha_J+\beta_J|}\geq\sup_{t\in\R^d}
|t^{\alpha_J+\beta_J}e^{-\omega_\bfM(t/C)}|
\geq \gamma^{\frac{\alpha_J+\beta_J}{2}}e^{-\omega_\bfM(\gamma^{1/2}/C)},
\eeqs
taking $t=\gamma^{1/2}$.

If we replace \eqref{11C} and \eqref{12C} in \eqref{10C} we finally get
\beqsn
\|x^\alpha\partial^\beta H_\gamma\|_2\leq B(2HC)^{|\alpha+\beta|}
N_{\alpha+\beta}e^{\omega_\bfM(\gamma^{1/2}/C)}.
\eeqsn
\end{proof}

\begin{Prop}
\label{cor36G}
Let $\M$ be a weight matrix  that satisfies \eqref{12L2R} and
\eqref{M2'R} (\eqref{12L2B} and \eqref{M2'B}). Then
$H_\gamma\in\Sch_{\{\M\}}$ ($H_\gamma\in\Sch_{(\M)}$)
for all $\gamma\in\N_0^d$.
\end{Prop}

\begin{proof}
By Lemma~\ref{lemma35G}, if \eqref{12L2R} is satisfied, we have
\beqsn
\forall\lambda>0\,\exists\;\kappa\geq\lambda,\,\exists B,C,H>0:\
\|x^\alpha\partial^\beta H_\gamma\|_2\leq B(2HC)^{|\alpha+\beta|}M^{(\kappa)}_{|\alpha+\beta|}e^{\omega_{\bfM^{(\lambda)}}(\gamma^{1/2}/C)}.
\eeqsn
Hence $H_\gamma\in\Sch_{\{\M\}}$ by Proposition~\ref{lemma25G}. Similarly, in the Beurling case, if \eqref{12L2B} is satisfied, we obtain
\beqsn
&&\forall\lambda>0\,\exists\;0<\kappa\leq\lambda,\,\exists H>0:\, \forall C>0\,\exists B>0:\\
&&\|x^\alpha\partial^\beta H_\gamma\|_2\leq B(2HC)^{|\alpha+\beta|}M^{(\lambda)}_{|\alpha+\beta|}e^{\omega_{\bfM^{(\kappa)}}(\gamma^{1/2}/C)}.
\eeqsn
So $H_\gamma\in\Sch_{(\M)}$ by Proposition~\ref{lemma25G}.
\end{proof}
The next result gives information about the non-triviality of the classes $\Sch_{\{\M\}}$ and $\Sch_{(\M)}$. Indeed, we characterize  when the Hermite functions $H_{\gamma}$ are contained in such classes.

\begin{Prop}
\label{lemma37G}
Let $\M$ be a weight matrix  that satisfies
\eqref{M2'R}, \eqref{37LR}; then the following are equivalent:
\begin{itemize}
\item[$(a)$]
$\exists\lambda>0\,\exists C,C_1>0:\quad
\alpha^{\alpha/2}\leq C_1 C^{|\alpha|}M^{(\lambda)}_\alpha,\quad\forall\alpha\in\N_0^d$;
\item[$(b)$]
$\M$ satisfies \eqref{12L2R};
\item[$(c)$]
$H_\gamma\in\Sch_{\{\M\}}$ for all $\gamma\in\N_0^d$.
\end{itemize}

If $\M$ satisfies \eqref{M2'B}, \eqref{37LB}, then the following are equivalent:
\begin{itemize}
\item[$(a)'$]
$\forall\lambda,C>0\,\exists C_1>0:\quad
\alpha^{\alpha/2}\leq C_1 C^{|\alpha|}M^{(\lambda)}_\alpha,\quad\forall\alpha\in\N_0^d$;
\item[$(b)'$]
$\M$ satisfies \eqref{12L2B};
\item[$(c)'$]
$H_\gamma\in\Sch_{(\M)}$ for all $\gamma\in\N_0^d$.
\end{itemize}
\end{Prop}

\begin{proof}
The implications $(b)\Rightarrow(c)$ and $(b)'\Rightarrow(c)'$ follow from
Proposition~\ref{cor36G}. To see $(a)\Rightarrow(b)$, fix an arbitrary $\mu>0$ and $\lambda$ as in $(a)$. We have
\beqsn
\alpha^{\alpha/2}M_\beta^{(\mu)}\leq C_1 C^{\vert\alpha\vert} M_\alpha^{(\lambda)} M_\beta^{(\mu)}.
\eeqsn
So, for $\nu=\max\{\lambda,\mu\}$, by \eqref{defcalM} and \eqref{37LR}, there exists $\kappa\geq\nu$ and $A\geq 1$ such that
\beqsn
\alpha^{\alpha/2}M^{(\mu)}_\beta\leq C_1C^{|\alpha|}M^{(\nu)}_\alpha
M^{(\nu)}_\beta\leq C_1 C^{|\alpha|}A^{|\alpha+\beta|}M^{(\kappa)}_{\alpha+\beta},
\quad\alpha,\beta\in\N_0^d.
\eeqsn
Now, we prove $(a)'\Rightarrow(b)'$.
For any given $\lambda>0$, let $0<\kappa\leq\lambda$ and $A\geq1$ such that
\eqref{37LB} holds. By $(a)'$ applied to this $\kappa$, there is, for any $C>0$, some 
$C_1>0$ depending on $\kappa$ and $C$ such that
\beqsn
\alpha^{\alpha/2}M^{(\kappa)}_\beta\leq C_1C^{|\alpha|}M^{(\kappa)}_\alpha
M^{(\kappa)}_\beta\leq C_1 C^{|\alpha|}A^{|\alpha+\beta|}M^{(\lambda)}_{\alpha+\beta},
\quad\alpha,\beta\in\N_0^d.
\eeqsn
If $(c)$ holds, in particular, $H_0\in\Sch_{\{\M\}}$. Hence there exist some $C,h>0$ and $\lambda>0$ such that $\|x^{\alpha}H_0\|_{\infty}\le Ch^{|\alpha|}M^{(\lambda)}_{\alpha}$ for all $\alpha\in\N_0^d$. Taking  
$x=\alpha^{1/2}$, $\alpha\in\N_0^d$ arbitrary, yields $$|\alpha^{\alpha/2}H_0(\alpha^{1/2})|=\frac{1}{\pi^{d/4}}\alpha_1^{\alpha_1/2}e^{-\alpha_1/2}\cdots\alpha_d^{\alpha_d/2}e^{-\alpha_d/2}=\frac{1}{\pi^{d/4}}\alpha^{\alpha/2}e^{-|\alpha|/2}.$$
Hence $\alpha^{\alpha/2}\pi^{-d/4}e^{-|\alpha|/2}\le\|x^{\alpha}H_0\|_{\infty}\le Ch^{|\alpha|}M^{(\lambda)}_{\alpha}$ for all $\alpha\in\N_0^d$, which shows $(a)$. 

The Beurling case  $(c)'\Rightarrow(a)'$ is analogous since now, for any given $\lambda$ and $h>0,$ there exists $C_{\lambda,h}>0$ such that $\|x^{\alpha}H_0\|_{\infty}\le C_{\lambda,h}h^{|\alpha|}M^{(\lambda)}_{\alpha}$ for all $\alpha\in\N_0^d$.\end{proof}

\section{Matrix sequence spaces}
\label{sec4}

Let us consider, for $\bfM=(M_\alpha)_{\alpha\in\N_0^d}$, the following
sequence spaces in the Roumieu and the Beurling cases:
\beqsn
&&\Lambda_{\{\bfM\}}:=\{\bfc=(c_\alpha)\in \C^{\N_0^d}:\ \exists\, h>0,\ \
\|\bfc\|_{\bfM,h}:=\sup_{\alpha\in\N_0^d}|c_\alpha|e^{\omega_\bfM(\alpha^{1/2}/h)}
<+\infty\},\\
&&\Lambda_{(\bfM)}:=\{\bfc=(c_\alpha)\in \C^{\N_0^d}:\ \forall\, h>0,\ \
\|\bfc\|_{\bfM,h}<+\infty\}.
\eeqsn
Since $h\mapsto\omega_\bfM(\alpha^{1/2}/h)$ is decreasing we can also write
\beqsn
&&\Lambda_{\{\bfM\}}=\{\bfc=(c_\alpha)\in \C^{\N_0^d}:\ \exists\,j\in\N,\ \ 
\|\bfc\|_{\bfM,j}<+\infty\},\\
&&\Lambda_{(\bfM)}=\{\bfc=(c_\alpha)\in \C^{\N_0^d}:\ \forall\,j\in\N,\
\|\bfc\|_{\bfM,1/j}<+\infty\}.
\eeqsn
Indeed, it sufficies to take $j=[h]+1$ in the Roumieu case and $j=\left[\frac1h\right]+1$
in the Beurling case.

Now, for a weight matrix $\M$ as in \eqref{defcalM} we denote
\beqsn
&&\Lambda_{\{\M\}}:=
\bigcup_{\lambda>0}\Lambda_{\{\bfM^{(\lambda)}\}}
=\{\bfc=(c_\alpha)\in \C^{\N_0^d}:\ \exists\,\lambda,h>0,\ \
\|\bfc\|_{\bfM^{(\lambda)},h}<+\infty\},\\
&&\Lambda_{(\M)}:=
\bigcap_{\lambda>0}\Lambda_{(\bfM^{(\lambda)})}
=\{\bfc=(c_\alpha)\in \C^{\N_0^d}:\ \forall\, \lambda,h>0,\ \
\|\bfc\|_{\bfM^{(\lambda)},h}<+\infty\}.
\eeqsn
Since $\bfM^{(\lambda)}\leq\bfM^{(\kappa)}$ for $0<\lambda\leq\kappa$ by assumption,
then $\omega_{\bfM^{(\lambda)}}\geq\omega_{\bfM^{(\kappa)}}$.
Moreover $h\mapsto e^{\omega_{\bfM^{(\kappa)}}(\alpha^{1/2}/h)}$ is
decreasing for all $\kappa>0$, $\alpha\in\N_0^d$.
It follows that we can write $\Lambda_{\{\M\}}$ ($\Lambda_{(\M)}$) as
inductive (projective) limit:
\beqs
\label{matrix-R}
&&\Lambda_{\{\M\}}
=\{\bfc=(c_\alpha)\in \C^{\N_0^d}:\ \exists j\in\N,\ \
\|\bfc\|_{\bfM^{(j)},j}<+\infty\},\\
\label{matrix}
&&\Lambda_{(\M)}=\{\bfc=(c_\alpha)\in \C^{\N_0^d}:\ \forall j\in\N,\ \ 
\|\bfc\|_{\bfM^{(1/j)},1/j}<+\infty\}.
\eeqs

%
%

We observe that by Remark \ref{omegafinite} it seems natural to require that  $\lim_{|\alpha|\rightarrow\infty}(M_{\alpha})^{1/|\alpha|}=+\infty$ for the definition of $\Lambda_{\{\bfM\}}$ and $\Lambda_{(\bfM)}$. In fact, otherwise  $\omega_\bfM(t)=+\infty$ for all large $t\in\R^d$ and we   get $\Lambda_{(\bfM)}=\{0\}$ and $\Lambda_{\{\bfM\}}$ consisting of sequences having only finitely many values $\neq 0$.

However, in our next main result, by Remark \ref{omegafinite2} and assumption \eqref{12L2R} (\eqref{12L2B} respectively), we have the warranty of the finiteness of all associated weight functions under consideration.

\begin{Th}
\label{th41G}
Let $\M$ be a weight matrix  satisfying \eqref{12L2R} and
\eqref{M2'R}. Then the Hermite functions are an absolute
Schauder basis in $\Sch_{\{\M\}}$ and
\beqsn
T:\ \Sch_{\{\M\}}&&\longrightarrow \Lambda_{\{\M\}}\\
f&&\longmapsto (\xi_\gamma(f))_{\gamma\in\N_0^d}:=
\left(\int_{\R^d}f(x)H_\gamma(x)dx\right)_{\gamma\in\N_0^d}
\eeqsn
defines an isomorphism.

If $\M$ satisfies \eqref{12L2B} and \eqref{M2'B}, then the
Hermite functions are an absolute Schauder basis in $\Sch_{(\M)}$ and the above
defined operator $T:\ \Sch_{(\M)}\rightarrow \Lambda_{(\M)}$ is an isomorphism.
\end{Th}

\begin{proof}
%
By Proposition~\ref{lemma25G} we can assume that
$\Sch_{\{\M\}}$ and $\Sch_{(\M)}$ are  defined by $L^2$ norms.
First, we consider the Roumieu case. 
If $f\in\Sch_{\{\M\}}$,  there exist 
$\lambda,C,C_1>0$ such that
\beqsn
\|f\|_{2,\bfM^{(\lambda)},C}=:C_1<+\infty.
\eeqsn
By \eqref{15C} and Lemma~\ref{cor34G}, there exists $\kappa\geq\lambda$,
$B,C,H>0$ such that for all $\gamma,\alpha\in\N_0^d$, since $\|H_\gamma\|_2=1$ for all $\gamma\in\N_0^d$, we have 
\beqsn
|\xi_\gamma(f)|^2\gamma^\alpha&&=|\langle f,H_\gamma\rangle|^2\gamma^\alpha
\leq |\langle f,\sqrt{2^{|\alpha|}\gamma^\alpha}H_\gamma\rangle|^2
=|\langle f,A_-^\alpha(H_{\gamma+\alpha})\rangle|^2\\
=&&|\langle A_+^\alpha(f),H_{\gamma+\alpha}\rangle|^2\leq
\|A_+^\alpha(f)\|_2^2\|H_{\gamma+\alpha}\|_2^2
\leq C_1^2B^2e^d(9\sqrt{2}HC)^{2|\alpha|}(M^{(\kappa)}_\alpha)^2.
\eeqsn
Therefore, by definition of the associated weight function, and using the notation of \eqref{4}, since $|(\gamma^{1/2})^\alpha|=|\gamma_1^{\alpha_1/2}\cdots\gamma_d^{\alpha_d/2}|
=(\gamma^\alpha)^{1/2}$, we obtain
\beqsn
|\xi_\gamma(f)|e^{\omega_{\bfM^{(\kappa)}}(\gamma^{1/2}/(9\sqrt{2}HC))}
=&&\sup_{\alpha\in\N_{0,\gamma}^d} \frac{|\xi_\gamma(f)|\left|\left(\frac{\gamma^{1/2}}{9\sqrt{2}HC}\right)^\alpha\right|}{M^{(\kappa)}_\alpha}
\leq
C_1B e^{d/2}.
\eeqsn
Hence $(\xi_\gamma(f))_\gamma\in\Lambda_{\{\M\}}$ and, more precisely, there
exist $\kappa\geq\lambda$, $H,C>0$ and $B\geq1$ such that
\beqs
\label{C40}
\|(\xi_\gamma(f))_\gamma\|_{\bfM^{(\kappa)},9\sqrt{2}HC}\leq B e^{d/2}
\|f\|_{2,\bfM^{(\lambda)},C}\,.
\eeqs
This proves that $T$ is continuous in the Roumieu case \cite[Proposition 24.7]{MV}. 

On the other hand, given $\bfc=(c_\gamma)_{\gamma\in\N_0^d}\in\Lambda_{\{\M\}}$, let
$\lambda,C^*>0$ such that
\beqsn
\sup_{\gamma\in\N_0^d}|c_\gamma|e^{\omega_{\bfM^{(\lambda)}}(\gamma^{1/2}/C^*)}
=\|\bfc\|_{\bfM^{(\lambda)},C^*}=:C_1^*<+\infty.
\eeqsn
By Lemma~\ref{lemma21G}, there exist $\kappa\geq\lambda$ and $B_1,B_2\geq1$ such
that
\beqsn
e^{-\omega_{\bfM^{(\lambda)}}(B_2t)+\omega_{\bfM^{(\kappa)}}(t)}
\leq B_1(1+|t|)^{-2(d+1)},\quad t\in\R^d.
\eeqsn
Then, by \eqref{12L2R}, there exist $\kappa'\geq\kappa$ and $B,C,H>0$
with $C\geq B_2C^*$, such that, by Lemma~\ref{lemma35G},
\beqs
\nonumber
|c_\gamma|\cdot\|x^\alpha\partial^\beta H_\gamma\|_2\leq&&|c_\gamma|(2HC)^{|\alpha+\beta|}
M^{(\kappa')}_{\alpha+\beta}Be^{\omega_{\bfM^{(\kappa)}}(\gamma^{1/2}/C)}\\
\nonumber
\leq&&C_1^*B(2HC)^{|\alpha+\beta|}M^{(\kappa')}_{\alpha+\beta}
e^{-\omega_{\bfM^{(\lambda)}}(\gamma^{1/2}/C^*)
+\omega_{\bfM^{(\kappa)}}(\gamma^{1/2}/(B_2C^*))}\\
\label{16C}
\leq&&C_1^*BB_1(2HC)^{|\alpha+\beta|}
M^{(\kappa')}_{\alpha+\beta}
\left(1+\left|\frac{\gamma^{1/2}}{B_2C^*}\right|\right)^{-2(d+1)}.
\eeqs
Since here $|\gamma^{1/2}|$ denotes the Euclidean norm of the multi-index $\gamma^{1/2}$, we have
\beqs
\label{dis3}
|\gamma^{1/2}|^{2(d+1)}=(\gamma_1+\ldots+\gamma_d)^{d+1}
\ge |\gamma|^{d+1}.
\eeqs
Hence
\beqsn
\sum_{\gamma\in\N_0^d}|c_\gamma|\cdot\|x^\alpha\partial^\beta H_\gamma\|_2\leq&&
C_1^*BB_1(2HC)^{|\alpha+\beta|}M^{(\kappa')}_{\alpha+\beta}
\sum_{\gamma\in\N_0^d}\frac{1}{\left(1+\left|\frac{\gamma^{1/2}}{B_2C^*}\right|\right)^{2(d+1)}}\\
\leq&&C_1^*BB_1(2HC)^{|\alpha+\beta|}M^{(\kappa')}_{\alpha+\beta}
\sum_{\gamma\in\N_0^d}\frac{1}{1+\left|\frac{\gamma^{1/2}}{B_2C^*}\right|^{2(d+1)}}\\
=&&C_1^*BB_1(2HC)^{|\alpha+\beta|}M^{(\kappa')}_{\alpha+\beta}
\sum_{\gamma\in\N_0^d}\frac{(B_2C^*)^{2(d+1)}}{(B_2C^*)^{2(d+1)}+|\gamma|^{d+1}}.
\eeqsn
Hence
\beqs
\label{C41}
\left\|\sum\nolimits_{\gamma\in\N_0^d}c_\gamma H_\gamma\right\|_{2,\bfM^{(\kappa')},2HC}
\leq C_1^*BB_1\tilde{C}
=BB_1\tilde{C}\|\bfc\|_{\bfM^{(\lambda)},C^*},
\eeqs
for $\tilde{C}=\sum_{\gamma\in\N_0^d}
(B_2C^*)^{2(d+1)}/((B_2C^*)^{2(d+1)}+|\gamma|^{d+1})<+\infty$. This shows that $T^{-1}$ is continuous and, moreover, that $(H_\gamma)_\gamma$ is an absolute Schauder basis in $\Sch_{\{\M\}}$.

Let now $f\in\Sch_{(\M)}$ and $\lambda,C>0$ be given. We  consider $0<\kappa\leq\lambda$,
$H,B>0$ as in \eqref{12L2B} (with $\kappa$ and $H$ depending only on $\lambda$) and we set
\beqsn
C_1:=\|f\|_{2,\bfM^{(\kappa)},C}<+\infty.
\eeqsn
By Lemma~\ref{lemma33G}, we have 
\beqsn
\|A_+^\alpha f\|_2\leq C_1B e^{d/2}(9\sqrt{2}HC)^{|\alpha|}M^{(\lambda)}_\alpha,
\quad\alpha\in\N_0^d.
\eeqsn
Hence, proceeding as in the Roumieu case, we deduce that, for all $\lambda,C>0$, there
exist $0<\kappa\leq\lambda$ and $B,H>0$ such that
\beqs
\label{C42}
\|(\xi_\gamma(f))_\gamma\|_{\bfM^{(\lambda)},9\sqrt{2}HC}\leq B e^{d/2}
\|f\|_{2,\bfM^{(\kappa)},C}.
\eeqs
This shows that $T$ is continuous in the Beurling case.

Now, if $\bfc=(c_\gamma)_{\gamma\in\N_0^d}\in\Lambda_{(\M)}$, then
by \eqref{12L2B} and Lemma~\ref{lemma35G}, for all $\lambda,C>0$ there exist
$0<\kappa\leq\lambda$, and $H,B>0$ (with $\kappa$ and $H$ depending only on $\lambda$) such that
\beqsn
\|x^\alpha\partial^\beta H_\gamma\|_2\leq(2HC)^{|\alpha+\beta|}
M^{(\lambda)}_{\alpha+\beta}Be^{\omega_{\bfM^{(\kappa)}}(\gamma^{1/2}/C)}.
\eeqsn
By Lemma~\ref{lemma21G}, there exist  $0<\kappa'\leq\kappa$ and $B_1,B_2\geq1$
such that
\beqsn
e^{-\omega_{\bfM^{(\kappa')}}(B_2t)+\omega_{\bfM^{(\kappa)}}(t)}
\leq B_1(1+|t|)^{-2(d+1)}, \ \ t\in \R^d.
\eeqsn
Since $\bfc\in\Lambda_{(\M)}$, we have 
\beqsn
\sup_{\gamma\in\N_0^d}|c_\gamma|e^{\omega_{\bfM^{(\kappa')}}(B_2\gamma^{1/2}/C)}
=\|\bfc\|_{\bfM^{(\kappa')},C/B_2}=:C_1<+\infty.
\eeqsn
Therefore, arguing as in the Roumieu case,
\beqsn
\sum_{\gamma\in\N_0^d}|c_\gamma|\cdot\|x^\alpha\partial^\beta H_\gamma\|_2\leq&&
C_1B(2HC)^{|\alpha+\beta|}M^{(\lambda)}_{\alpha+\beta}
\sum_{\gamma\in\N_0^d}
e^{-\omega_{\bfM^{(\kappa')}}(B_2\gamma^{1/2}/C)+\omega_{\bfM^{(\kappa)}}(\gamma^{1/2}/C)}\\
\leq&&C_1BB_1(2HC)^{|\alpha+\beta|}M^{(\lambda)}_{\alpha+\beta}
\sum_{\gamma\in\N_0^d}\frac{1}{(1+|\gamma^{1/2}/C|)^{2(d+1)}}\\
\leq&&\tilde{B}C_1(2HC)^{|\alpha+\beta|}M^{(\lambda)}_{\alpha+\beta},
\eeqsn
for $\tilde{B}=BB_1\sum_{\gamma\in\N_0^d}C^{2(d+1)}/(C^{2(d+1)}+|\gamma|^{d+1})
<+\infty$.  For all $\lambda,h>0$ there exist then $\kappa'\leq\lambda$
and $\tilde{h}=h/(2HB_2)>0$ such that
\beqs
\label{C43}
\left\|\sum\nolimits_{\gamma\in\N_0^d}c_\gamma H_\gamma\right\|_{2,\bfM^{(\lambda)},h}
\leq\tilde{B}\|\bfc\|_{\bfM^{(\kappa')},\tilde h}.
\eeqs
This shows that $T^{-1}$ is continuous on $\Sch_{(\M)}$ and that $(H_\gamma)_\gamma$ is an absolute  Schauder basis in $\Sch_{(\M)}$, which finishes the proof.
\end{proof}

As in \cite[Corollary 3.6]{L} we also have that the Fourier transform is well adapted to our spaces and it is an isomorphism:
\begin{Cor}\label{Fourier-iso}
Let $\M$ be a weight matrix satisfying \eqref{12L2R} and
\eqref{M2'R} (\eqref{12L2B} and \eqref{M2'B}). Then the Fourier transform is an isomorphism in $\Sch_{\{\M\}}(\R^d)$
($\Sch_{(\M)}(\R^d)$).
\end{Cor}

Now, we prove that the spaces of sequences are nuclear.

\begin{Th}
\label{thm55G}
Let $\M=(M^{(\lambda)}_\alpha)_{\lambda>0,\alpha\in\N_0^d}$ be a weight matrix satisfying \eqref{M2'B}. Then $\Lambda_{(\M)}$ is nuclear.
\end{Th}

\begin{proof}
By \eqref{matrix} and
%
\cite[Theorem 3.1]{BJOS}, the sequence space $\Lambda_{(\M)}$ is nuclear
if and only if
\beqs
\label{series}
\forall j\in\N\,\exists\ell\in\N,\ell>j: \ \ \ 
\sum_{\gamma\in\N^d_0} e^{\omega_{\bfM^{(1/j)}}(j\gamma^{1/2})-\omega_{\bfM^{(1/\ell)}}(\ell \gamma^{1/2})}
<+\infty.
\eeqs
Moreover, by Lemma~\ref{lemma58G} condition \eqref{5BJO-Rodino} is satisfied.
We can thus proceed as in the proof of Theorem~1 of \cite{BJO-Rodino} to
prove that \eqref{5BJO-Rodino} implies that the series in \eqref{series} converges, and hence
$\Lambda_{(\M)}$ is nuclear.
To this aim we fix an index $\lambda>0$ and $N\in\N$ with $N>2d$ and remark that if the inequality \eqref{5BJO-Rodino} holds for $\lambda=1/j$ and $\kappa\le\lambda$, then it holds also if, instead of $\kappa$, we put
$\kappa'=1/h$ with $h\in\N$, $h>[\frac{1}{\kappa}]+1$, since
$\bfM^{(\kappa')}\leq\bfM^{(\kappa)}$ for $\kappa'\leq\kappa$ and hence
$\omega_{\bfM^{(\kappa)}}\leq\omega_{\bfM^{(\kappa')}}$.
Then, for $\ell\geq Ah$ (so that $\ell\geq Aj$ and $\ell\geq h>j$ and note that the constant $A$ is also depending on the chosen $N$):
\beqsn
&&\sum_{\gamma\in\N^d_0}e^{\omega_{\bfM^{(1/j)}}(j\gamma^{1/2})-\omega_{\bfM^{(1/\ell)}}(\ell \gamma^{1/2})}
\leq\sum_{\gamma\in\N^d_0\backslash\{0\}}e^{\omega_{\bfM^{(1/j)}}(j\gamma^{1/2})-\omega_{\bfM^{(1/h)}}(A j\gamma^{1/2})}+1\\
&&\ \ \ \ \ \ \ \ \ \ \ \ \ \ \  \le\sum_{\gamma\in\N^d_0\backslash\{0\}}e^{-N\log|j\gamma^{1/2}|+B}+1=e^Bj^{-N}\sum_{\gamma\in\N^d_0\backslash\{0\}}
\frac{1}{\vert\gamma\vert^{N/2}}+1<+\infty,
\eeqsn
by our choice of $N>2d$. This completes the proof.
\end{proof}

Concerning the Roumieu case we have the following result.

\begin{Th}
\label{thm55G-R}
Let $\M=(M^{(\lambda)}_\alpha)_{\lambda>0,\alpha\in\N_0^d}$ be a weight matrix  satisfying \eqref{M2'R}. Then $\Lambda_{\{\M\}}$ is nuclear.
\end{Th}

\begin{proof}
For
\beqsn
a_{\alpha,j}:= e^{-\omega_{\bfM^{(j)}}(\alpha^{1/2}/j)},
\eeqsn
we consider the matrices
\beqsn
A:=\left( a_{\alpha,j}\right)_{\alpha\in\N_0^d,\ j\in\N},\qquad V:=\left( v_{\alpha,j}\right)_{\alpha\in\N_0^d,\ j\in\N}\ \text{with}\ v_{\alpha,j}=a_{\alpha,j}^{-1}.
\eeqsn
We observe that $A$ is a K\"{o}the matrix since its entries are strictly positive and $a_{\alpha,j}\le a_{\alpha,j+1}$ for every $j\in\N$. We consider now the space
\beqsn
\lambda_{(\mathcal{M})}:=\{{\bf c}=(c_\alpha)\in \C^{\N_0^d}:\ \forall j\in\N,\ \sum_{\alpha\in\N_0^d} |c_\alpha| a_{\alpha,j}<\infty\}.
\eeqsn
Since $\N_0^d = \cup_{m\in\N}I_m$ with $I_m=\{\alpha\in\N_0^d : |\alpha|\leq m\}$
and $v_{\alpha,j}>0$ for every $\alpha$ and $j$, we have that the matrix $V$ satisfies the \emph{condition (D)} of \cite{BM} (see also \cite{BB}). From \cite[Theorem 18(1)]{BB} we have that $\lambda_{(\mathcal{M})}$ is distinguished, and then, from \cite[Corollary 8{\it (f)}]{BB} and \eqref{matrix-R} we get 
\beqsn
\left(\lambda_{(\mathcal{M})}\right)_b' = \Lambda_{\{\mathcal{M}\}}.
\eeqsn
Since a Fr\'echet space is nuclear if and only if its dual is nuclear
\cite[pg.~78]{Pi}, it is enough to prove that $\lambda_{(\mathcal{M})}$ is nuclear; from \cite[Theorem 3.1]{BJOS} this is true if and only if
\beqs
\label{series-R}
\forall k\in\N\,\exists m\in\N,m>k:\ \
\sum_{\gamma\in\N^d_0} e^{\omega_{\bfM^{(m)}}(\gamma^{1/2}/m)-\omega_{\bfM^{(k)}}(\gamma^{1/2}/k)}
<+\infty.
\eeqs
By Lemma~\ref{lemma58G} we can now use \eqref{5BJO-Rodino1} with $\lambda=k$ and with a fixed $N>2d$; since $\omega_{\bfM^{(m)}}(t)\leq \omega_{\bfM^{(\kappa)}}(t)$ for every $m\geq \kappa$, we can replace in \eqref{5BJO-Rodino1} $\kappa$ by $m=\max\{\kappa,Ak\}$, obtaining that for every $k\in\N$ there exists $m\geq k$ such that
\beqsn
\omega_{\bfM^{(m)}}\left(\frac{\gamma^{1/2}}{m}\right)+N\log\left\vert\frac{\gamma^{1/2}}{m}\right\vert\leq \omega_{\bfM^{(k)}}\left(A\frac{\gamma^{1/2}}{m}\right)+B,
\eeqsn
for every $\gamma\neq 0$. Since $A\leq m/k$ we obtain
\beqsn
e^{\omega_{\bfM^{(m)}}(\gamma^{1/2}/m)-\omega_{\bfM^{(k)}}(\gamma^{1/2}/k)}\leq e^B m^N\frac{1}{|\gamma^{1/2}|^N}\leq e^B m^N\frac{1}{|\gamma|^{N/2}},
\eeqsn
for $\gamma\neq 0$;
since $N>2d$ we have that \eqref{series-R} holds, and then by estimating as in Theorem \ref{thm55G} the proof is complete.
\end{proof}


\begin{Cor}
\label{lemma57Gend}
If $\M=(M^{(\lambda)}_\alpha)_{\lambda>0,\alpha\in\N_0^d}$ is a weight matrix
satisfying \eqref{12L2B} and \eqref{M2'B}, then the space $\Sch_{(\M)}$ is nuclear.
If $\M$ satisfies \eqref{12L2R} and \eqref{M2'R}, then the space $\Sch_{\{\M\}}$ is nuclear.

\end{Cor}
\begin{proof}
The Beurling case follows from Theorems~\ref{th41G} and \ref{thm55G}, and the Roumieu case follows from Theorems~\ref{th41G} and \ref{thm55G-R}.
\end{proof}

\begin{Prop}
\label{lemma57G}
Let $\M=(M^{(\lambda)}_p)_{\lambda>0,p\in\N_0}$ be a weight matrix (with $d=1$), such that each sequence $\bfM^{(\lambda)}$ satisfies \eqref{M1} and $\lim_{p\rightarrow\infty}(M_p)^{1/p}=+\infty$. Assume, moreover, that
\beqsn
\mu^{(\lambda)}_p:=\frac{M^{(\lambda)}_p}{M^{(\lambda)}_{p-1}},\qquad p\in\N,
\eeqsn
satisfies $\mu^{(\lambda)}\leq\mu^{(\kappa)}$ for all $0<\lambda\leq\kappa$ and
$\mu^{(\lambda)}_0=1$ for all $\lambda>0$.
Then the following conditions are equivalent:
\begin{itemize}
\item[(a)]
$\forall j\in\N\ \exists\,\ell\in\N,\ell>j:\ \ 
\sum_{k=1}^{+\infty}
e^{\omega_{\bfM^{(1/j)}}(jk^{1/2})-\omega_{\bfM^{(1/\ell)}}(\ell k^{1/2})}<+\infty$;
\item[(b)]
$\forall\;\lambda>0\ \exists\;0<\kappa<\lambda, A\geq1\ \forall\,p\in\N: \ \
M^{(\kappa)}_{p+1}\leq A^{p+1}M^{(\lambda)}_p.$
\end{itemize}
\end{Prop}

\begin{proof}
If condition $(b)$ is satisfied, then \eqref{M2'B} is satisfied and hence also condition $(a)$, as we already saw in the proof of Theorem~\ref{thm55G}.

Let us now assume condition $(a)$ and prove $(b)$.
To this aim let us first remark that
\beqs
\label{C20}
k\longmapsto\omega_{\bfM^{(1/j)}}(jk^{1/2})-\omega_{\bfM^{(1/\ell)}}(\ell k^{1/2})
\eeqs
is decreasing. Indeed,
\beqsn
\omega_{\bfM^{(1/\ell)}}(\ell k^{1/2})-\omega_{\bfM^{(1/j)}}(j k^{1/2})=&&
\left(\omega_{\bfM^{(1/\ell)}}(\ell k^{1/2})-\omega_{\bfM^{(1/\ell)}}(j k^{1/2})\right)\\
&&+\left(\omega_{\bfM^{(1/\ell)}}(j k^{1/2})-\omega_{\bfM^{(1/j)}}(j k^{1/2})\right)
=:\omega_1+\omega_2.
\eeqsn

The first difference $\omega_1=\omega_{\bfM^{(1/\ell)}}(\ell k^{1/2})-\omega_{\bfM^{(1/\ell)}}(j k^{1/2})$ is increasing since by definition $t\mapsto \omega_{\bfM^{(1/\ell)}}(e^t)$ is convex (see the proof of Theorem~1 in \cite{BJO-Rodino} for the implication that the convexity implies that $\omega_1$ is increasing).

To prove that also the second difference $\omega_2$ is increasing, we set
\beqsn
\Sigma_{\bfM^{(\lambda)}}(t):=\#\{p\in\N:\ \mu^{(\lambda)}_p\leq t\}
\eeqsn
and remark that, by \eqref{M1} (see \cite[formula(3.11)]{K}),
\beqsn
\omega_{\bfM^{(\lambda)}}(t)=\int_0^t\frac{\Sigma_{\bfM^{(\lambda)}}(s)}{s}ds.
\eeqsn

Then
\beqsn
\omega_{\bfM^{(1/\ell)}}(t)-\omega_{\bfM^{(1/j)}}(t)
=\int_0^t\frac{\Sigma_{\bfM^{(1/\ell)}}(s)-\Sigma_{\bfM^{(1/j)}}(s)}{s}ds
\eeqsn
is an increasing function of $t$ since
\beqsn
\Sigma_{\bfM^{(1/\ell)}}(s)\geq \Sigma_{\bfM^{(1/j)}}(s),\qquad \ell>j,
\eeqsn
by the assumption $\mu^{(1/\ell)}_p\leq\mu^{(1/j)}_p$ for $\ell>j$.

Therefore $\omega_1$ and $\omega_2$ are increasing and we have thus proved that
\eqref{C20} is decreasing.
This condition together with assumption $(a)$ implies that
\beqsn
\lim_{k\to+\infty}k
e^{\omega_{\bfM^{(1/j)}}(jk^{1/2})-\omega_{\bfM^{(1/\ell)}}(\ell k^{1/2})}=0.
\eeqsn
There exists then
$A\geq1$ such that
\beqsn
\sup_{k\in\N}ke^{\omega_{\bfM^{(1/j)}}(jk^{1/2})-\omega_{\bfM^{(1/\ell)}}(\ell k^{1/2})}\leq A
\eeqsn
and hence, for all $k\in\N$,
\beqsn
\omega_{\bfM^{(1/j)}}(jk^{1/2})-\omega_{\bfM^{(1/\ell)}}(\ell k^{1/2})
\leq-\log k+\log A\leq -\log(jk^{1/2})+\log(jA).
\eeqsn
Choosing, for every $t\geq1$,  the smallest $k\in\N$ such that
$jk^{1/2}\in[t,(j+1)t]$, we finally have
\beqs
\nonumber
\omega_{\bfM^{(1/j)}}(t)+\log t\leq&&\omega_{\bfM^{(1/j)}}(jk^{1/2})+\log(jk^{1/2})\\
\nonumber
\leq&&\omega_{\bfM^{(1/\ell)}}(\ell k^{1/2})+\log(jA)
\label{19dec}
\leq\omega_{\bfM^{(1/\ell)}}\left(\frac{\ell}{j}(j+1) t\right)+\log(jA).
\eeqs
Since \eqref{19dec} is trivial for $0<t\leq1$, we have
proved that condition $(ii)$ of Lemma~\ref{lemma23G} is satisfied for
$\bfN=\bfM^{(1/j)}$ and $\bfM=\bfM^{(1/\ell)}$ and hence, from $(i)$ of
Lemma~\ref{lemma23G}, there exists $\tilde A\geq1$ such that
\beqsn
M^{(1/\ell)}_{p+1}\leq\tilde{A}^{p+1}M^{(1/j)}_p,\qquad\forall p\in\N_0.
\eeqsn
Then, for all $\lambda>0$, choosing $j\in\N$ so that $\frac1j\leq\lambda$, there exists
$\kappa=\frac1\ell<\frac1j\leq\lambda$  such that
condition $(b)$ holds.
\end{proof}

Proposition~\ref{lemma57G} yields now the following result.

\begin{Th}
\label{lemma57Gend2}
Let $\M=(M^{(\lambda)}_p)_{\lambda>0,p\in\N_0}$ be a weight matrix as in
Proposition~\ref{lemma57G}. Then the space $\Lambda_{(\M)}$ is nuclear if and only if condition
\eqref{M2'B} is satisfied.
\end{Th}

\begin{proof}
It follows from Theorem~\ref{thm55G} and, in particular, \eqref{series}.
\end{proof}

\begin{Th}
Let $\M=(M^{(\lambda)}_p)_{\lambda>0,p\in\N_0}$ be a weight matrix as in
Proposition~\ref{lemma57G}. Then the space $\Lambda_{\{\M\}}$ is nuclear if and only if condition
\eqref{M2'R} is satisfied.
\end{Th}

\begin{proof}
By the proof of Theorem~\ref{thm55G-R} we have that 
$\Lambda_{\{\M\}}$ is nuclear if and only if \eqref{series-R} is satisfied, and this is equivalent to \eqref{M2'R} since, analogously as in Proposition~\ref{lemma57G}, the
following two conditions are equivalent:
\begin{itemize}
\item[$(a)'$]
$\forall\, j\in\N\ \exists\,\ell\in\N,\ell>j:\ \
\sum_{k=1}^{+\infty}
e^{\omega_{\bfM^{(\ell)}}(k^{1/2}/\ell)-\omega_{\bfM^{(j)}}(k^{1/2}/j)}<+\infty$,
\item[$(b)'$]
$\forall\,\lambda>0\ \exists\,\kappa>\lambda, A\geq1\ \forall\,p\in\N:\ \
M^{(\lambda)}_{p+1}\leq A^{p+1}M^{(\kappa)}_p$.
\end{itemize}
Indeed, $(b)'$ implies \eqref{M2'R} and hence $(a)'$, i.e. \eqref{series-R},
in the one-dimensional case, by the proof of Theorem~\ref{thm55G-R}.

Conversely, if $(a)'$ holds then for every fixed $j\in\N$, 
and $\ell>j$ as in $(a)'$, there exists $A>\ell$ such that
\beqsn
\sup_{k\in\N}ke^{\omega_{\bfM^{(\ell)}}(k^{1/2}/\ell)-
\omega_{\bfM^{(j)}}(k^{1/2}/j)}\leq A,
\eeqsn
since
\beqsn
k\longmapsto\omega_{\bfM^{(\ell)}}(k^{1/2}/\ell)-\omega_{\bfM^{(j)}}(k^{1/2}/j)
\eeqsn
is decreasing, similarly as in the proof of Proposition~\ref{lemma57G}.
Then, for all $k\in\N$,
\beqsn
\omega_{\bfM^{(\ell)}}(k^{1/2}/\ell)-\omega_{\bfM^{(j)}}(k^{1/2}/j)
\leq-\log k+\log A\leq -\log(k^{1/2}/\ell)+\log(A/\ell).
\eeqsn
If $t\geq1$ we can choose a smallest $k\in\N$ such that
$k^{1/2}/\ell\in[t,(1+\frac1\ell)t]$ and obtain that
\beqs
\nonumber
\omega_{\bfM^{(\ell)}}(t)+\log t\leq&&\omega_{\bfM^{(\ell)}}(k^{1/2}/\ell)+\log(k^{1/2}/\ell)\\
\nonumber
\leq&&\omega_{\bfM^{(j)}}(k^{1/2}/j)+\log(A/\ell)
\nonumber
\label{19dec3}
\leq\omega_{\bfM^{(j)}}\left(\frac{\ell}{j}\left(1+\frac1\ell\right) t\right)+\log(A/\ell).
\eeqs
Since \eqref{19dec3} is trivial for $0<t\leq1$, we have that
\beqsn
\omega_{\bfM^{(\ell)}}(t)+\log t\leq\omega_{\bfM^{(j)}}(At)+B,\qquad\forall t>0,
\eeqsn
for $A=\frac\ell j\left(1+\frac1\ell\right)\geq1$ and $B=\log(A/\ell)>0$.
By Lemma~\ref{lemma23G} with $\bfM=\bfM^{(j)}$ and
$\bfN=\bfM^{(\ell)}$, for every $\lambda>0$ we can choose $j\in\N$, $j\geq\lambda$
so that $(b)'$ is satisfied for $\kappa=\ell>j\geq\lambda$.
The proof is complete.
\end{proof}

\section{Rapidly decreasing ultradifferentiable functions}
\label{sec5}

We shall now consider weight functions $\omega$ defined as below:
\begin{Def}
\label{defomega}
A {\em weight function} is a continuous increasing function
$\omega:\ [0,+\infty)\to[0,+\infty)$ such that
\begin{itemize}
\item[$(\alpha)$]
$\exists L\geq1\ \forall t\geq0:\ \omega(2t)\leq L(\omega(t)+1)$;
\item[$(\beta)$]
$\omega(t)=O(t^2)$ as $t\to+\infty$;
\item[$(\gamma)$]
$\log t=o(\omega(t))$ as $t\to+\infty$;
\item[$(\delta)$]
$\varphi_\omega(t):=\omega(e^t)$ is convex on $[0,+\infty)$.
\end{itemize}
Then we define $\omega(t):=\omega(|t|)$ if $t\in\R^d$.
\end{Def}

It is not restrictive to assume $\omega|_{[0,1]}\equiv0$.
As usual, we define the Young conjugate $\varphi^*_\omega$ of $\varphi_\omega$ by
\beqsn
\varphi^*_\omega(s):=\sup_{t\geq0}\{ts-\varphi_\omega(t)\},
\eeqsn
which is an increasing convex function such that $\varphi^{**}_\omega=\varphi_\omega$ and  $\varphi^*(s)/s$ is increasing \cite{Hcvx,BMT}. We remark that condition $(\beta)$ and the stronger condition
$\omega(t)=o(t^2)$ as $t$ tends to infinity are needed in the Roumieu and Beurling cases  for 
Corollaries~\ref{th54G} and \ref{cor56G}. On the other hand, condition $(\gamma)$ guarantees that
$\varphi^*_\omega$ is finite, so that, from
the properties of $\varphi^*_\omega$ (see \cite{BMT} or \cite[Lemma A.1]{BJO-PW}) we
easily obtain (cf. \cite{RS}):
\begin{Lemma}
\label{fromRS}
Let $\omega:\ [0,+\infty)\to[0,+\infty)$ be a weight function as in Definition~\ref{defomega}, and set
\beqs
\label{Wlambda}
W^{(\lambda)}_\alpha:=e^{\frac{1}{\lambda}\varphi^*_\omega(\lambda |\alpha|)},
\qquad\lambda>0,\alpha\in\N_0^d.
\eeqs
Then $W^{(\lambda)}_\alpha\in\R$  and the weight matrix
\beqs
\label{calMomega}
\M_\omega:=(\bfW^{(\lambda)})_{\lambda>0}=(W^{(\lambda)}_\alpha)_{\lambda>0,\,\alpha\in\N_0^d}
\eeqs
satisfies the following properties:
\begin{itemize}
\item[(i)]
$W^{(\lambda)}_0=1,\quad\lambda>0$;
\item[(ii)]
$(W^{(\lambda)}_\alpha)^2\leq W^{(\lambda)}_{\alpha-e_i}W^{(\lambda)}_{\alpha+e_i},\quad
\lambda>0,\alpha\in\N^d_0$ with $\alpha_i\neq 0$, and $i=1,\dots,d$;
\item[(iii)]
$\bfW^{(\kappa)}\leq\bfW^{(\lambda)},\quad 0<\kappa\leq\lambda$;
\item[(iv)]
$W^{(\lambda)}_{\alpha+\beta}\leq W^{(2\lambda)}_\alpha W^{(2\lambda)}_\beta,\quad\lambda>0,\alpha,\beta\in\N_0^d$;
\item[(v)]
$\forall h>0\ \exists A\geq1\ \forall\lambda>0\ \exists D\geq1\ \forall \alpha\in\N_0^d:\ \ 
\ h^{|\alpha|}W^{(\lambda)}_\alpha\leq DW^{(A\lambda)}_\alpha;$
\item[(vi)] Both conditions \eqref{M2'R} and \eqref{M2'B} are valid;
\item[(vii)]
Conditions \eqref{37LR} and \eqref{37LB} are satisfied for $\kappa=\lambda$ and
$A=1$.
\end{itemize}
\end{Lemma}

\begin{proof}
Let us first remark that condition $(\gamma)$ of Definition~\ref{defomega} ensures that 
$W^{(\lambda)}_\alpha\in\R$ for all $\lambda>0$ and $\alpha\in\N_0^d$.
Condition (i) is trivial since $\varphi^*_\omega(0)=0$.
Condition (ii) follows from the convexity of $\varphi^*_\omega$:
\beqsn
e^{\frac{2}{\lambda}\varphi^*_\omega(\lambda |\alpha|)}=
e^{\frac{2}{\lambda}\varphi^*_\omega\left(\frac{\lambda(|\alpha|-1)+\lambda(|\alpha|+1)}{2}\right)}
\leq e^{\frac{1}{\lambda}\varphi^*_\omega(\lambda |\alpha-e_i|)}e^{\frac{1}{\lambda}
\varphi^*_\omega(\lambda |\alpha+e_i|)}.
\eeqsn
The monotonicity property (iii) is clear since $\varphi^*_\omega(s)/s$ is increasing. Properties (iv), (v) and (vii) follow from \cite[Lemma~A.1]{BJO-PW}. Indeed, from
\cite[Lemma A.1(ix)]{BJO-PW}
\beqsn
e^{\frac{1}{\lambda}\varphi^*_\omega(\lambda |\alpha+\beta|)}\leq
e^{\frac{1}{2\lambda}\varphi^*_\omega(2\lambda |\alpha|)
+\frac{1}{2\lambda}\varphi^*_\omega(2\lambda |\beta|)}.
\eeqsn
From \cite[Lemma~A.1(iv)]{BJO-PW} with $A=L^2+L$ and $B=L^2$, where $L$
is the constant of condition $(\alpha)$ of Definition~\ref{defomega},
\beqsn
h^{|\alpha|}e^{\frac{1}{\lambda}\varphi^*_\omega(\lambda |\alpha|)}
\leq\Lambda_{h,\lambda}e^{\frac{1}{\lambda'}\varphi^*_\omega(\lambda' |\alpha|)}
\eeqsn
for all $\lambda'\geq\lambda B^{[\log h+1]}$ and $\Lambda_{h,\lambda}:=
e^{\frac{1}{\lambda}\left(1+\frac1L\right)[\log h+1]}$.
From \cite[Lemma~A.1(ii)]{BJO-PW}
\beqsn
e^{\frac{1}{\lambda}\varphi^*_\omega(\lambda |\alpha|)+\frac{1}{\lambda}
\varphi^*_\omega(\lambda |\beta|)}\leq
e^{\frac{1}{\lambda}\varphi^*_\omega(\lambda |\alpha+\beta|)}.
\eeqsn
Finally, (vi) is an immediate consequence of (iv).
\end{proof}

Let us now define the spaces of rapidly decreasing $\omega$-ultradifferentiable
functions, in the Roumieu case
\beqsn
\Sch_{\{\omega\}}(\R^d):=&&\big\{f\in C^\infty(\R^d):\ \exists\lambda>0,C>0:\
\sup_{\alpha,\beta\in\N_0^d}\|x^\alpha\partial^\beta f\|_\infty
e^{-\frac{1}{\lambda}\varphi^*_\omega(\lambda |\alpha+\beta|)}\leq C\big\}\\
=&&\big\{f\in C^\infty(\R^d):\ \exists\lambda>0,C>0:\ \|f\|_{\infty,\bfW^{(\lambda)}}:=
\sup_{\alpha,\beta\in\N_0^d}\frac{\|x^\alpha\partial^\beta f\|_\infty}{W^{(\lambda)}_{\alpha+\beta}}\leq C\big\},
\eeqsn
and in the Beurling case
\beqsn
\Sch_{(\omega)}(\R^d):=&&\big\{f\in C^\infty(\R^d):\ \forall\lambda>0\,\exists C_\lambda>0:\
\|f\|_{\infty,\bfW^{(\lambda)}}\leq C_\lambda\big\}.
\eeqsn
From Lemma~\ref{fromRS}(iv) and (vii) (see also \cite[Thm. 4.8]{BJO-Wigner}):
\beqsn
\Sch_{\{\omega\}}(\R^d)=\big\{f\in C^\infty(\R^d):\ \exists\lambda>0,C>0:\
\sup_{\alpha,\beta\in\N_0^d}
\frac{\|x^\alpha\partial^\beta f\|_\infty}{W^{(\lambda)}_{\alpha}
W^{(\lambda)}_{\beta}}\leq C\big\}
\eeqsn
and
\beqsn
\Sch_{(\omega)}(\R^d)=\big\{f\in C^\infty(\R^d):\ \forall\lambda>0\,\exists C_\lambda>0:\
\sup_{\alpha,\beta\in\N_0^d}\frac{\|x^\alpha\partial^\beta f\|_\infty}{W^{(\lambda)}_{\alpha}
W^{(\lambda)}_{\beta}}\leq C_\lambda\big\}.
\eeqsn
We refer to \cite{BJO-Wigner,BJO-PW,GZ} for more equivalent seminorms on
$\Sch_{(\omega)}(\R^d)$, if $\omega(t)=o(t^2)$.

We can also insert $h^{|\alpha+\beta|}$ at the denominator (for some $h>0$ in the
Roumieu case and for all $h>0$ in the Beurling case) by Lemma~\ref{fromRS}(v).
In particular, we have the following
\begin{Prop}
\label{lemma514RS}
Let $\omega$ be a weight function  and $\M_\omega$ the
weight matrix defined in  \eqref{Wlambda}, \eqref{calMomega}. We have $\Sch_{\{\M_\omega\}}=\Sch_{\{\omega\}}(\R^d)$ and $\Sch_{(\M_\omega)}=\Sch_{(\omega)}(\R^d)$ 
and the equalities are also topological.
\end{Prop}

\begin{Rem}\label{remarkseparating}
	\begin{em}
We observe that for the weight function  $\omega(t)=\log^s(1+t)$, for some $s>1$, we have that $\Sch_{(\omega)}(\R)$ never equals
$\Sch_{(M_p)}(\R)$ for any sequence $(M_p)_{p\in\N_0}$. Hence, $\Sch_{(\omega)}(\R)$ cannot be defined with sequences as in \cite{L} when $(M_p)$ satisfies $(M0)$, $(M1)$ and $(M2)'$ (see \cite{BMM} for the definition of $(M0)$; $(M1)$ and $(M2)'$ are recalled in \eqref{M1} and \eqref{M2'}).  
		
Indeed, by
\cite[Example 20]{BMM},  $\E_{(\omega)}(\R)\neq\E_{(M_p)}(\R)$ for any
sequence $(M_p)$ as considered just above, where $\E_{(\omega)}(\R)$ and $\E_{(M_p)}(\R)$ are the spaces of ultradifferentiable functions defined by weights and sequences (for the definitions see \cite{BMM}). 
We fix a sequence $(M_p)$ and prove that
$\Sch_{(\omega)}(\R)\neq\Sch_{(M_p)}(\R)$. Clearly, we can assume that $(M_p)$ is
non-quasianalytic since the weight $\omega$ is non-quasi-analytic. In particular, $(M_p)$ satisfies $(M0)$ (see  \cite{BMM}, condition $(M3)'$, and use also $(M1)$). 
If $f\in\E_{(M_p)}(\R)\setminus\E_{(\omega)}(\R)$, then there are  a compact set 
$K\subseteq\R$ and $m\in\N$ such that
$$\sup_{j\in\N_0}\sup_{x\in K}| f^{(j)}(x)|e^{-m\varphi^*\left(\frac{j}{m}\right)}=+\infty.$$ 
Hence
\beqsn
\forall n\in\N\ \exists x_n\in K, j_n\in \N\ \mbox{such that }
|f^{(j_n)}(x_n)|\geq ne^{m\varphi^*\left(\frac{j_n}{m}\right)}.
\eeqsn
Since $K$ is compact we can assume that the sequence
$(x_n)$ converges to some $x_0\in K$. Let $\varphi\in\D_{(M_p)}(\R)$ (the space of functions in $\E_{(M_p)}(\R)$ with compact support) with
$\varphi\equiv1$ in a neighbourhood of $x_0$.
Then $g=f\varphi\in\D_{(M_p)}(\R)\subseteq\Sch_{(M_p)}(\R)$ but,
for $n$ sufficiently large,
\beqsn
\frac{|g^{(j_n)}(x_n)|}{e^{m\varphi^*\left(\frac{j_n}{m}\right)}}=
\frac{|f^{(j_n)}(x_n)|}{e^{m\varphi^*\left(\frac{j_n}{m}\right)}}\geq n\longrightarrow+\infty,
\eeqsn
and hence $g\notin\Sch_{(\omega)}(\R)$ (see the definition of $\Sch_{(\omega)}(\R)$ above).
		
Analogously, for $f\in\E_{(\omega)}(\R) \setminus  \E_{(M_p)}(\R)$ we can construct
$g\in\Sch_{(\omega)}(\R) \setminus  \Sch_{(M_p)}(\R)$.
		
The same arguments are valid for the Roumieu case and for dimension bigger than one (considering always isotropic classes). 
\end{em}
\end{Rem}

The following Lemma was proved in dimension $1$ in \cite[Lemma~2.5]{JSS}; here we give a version of it in dimension $d$.
\begin{Lemma}
\label{lemma52G}
Let $\omega$ be a weight function. Then there exists a constant $B>0$ and, for every $\lambda>0$, there exists
$C_\lambda>0$, such that
\beqs
\label{4a}
\lambda\omega_{\bfW^{(\lambda)}}(t)\leq\omega(t)\leq B\lambda
\omega_{\bfW^{(\lambda)}}(t)+C_\lambda,\qquad t\in\R^d.
\eeqs
\end{Lemma}

\begin{proof}
For $t=0$ the thesis is trivial, so we can consider $t\neq 0$. Since $|t^\alpha|\leq |t|^{|\alpha|}$ for every multi-index $\alpha$, we have
\beqsn
\lambda\omega_{\bfW^{(\lambda)}}(t) &=& \lambda\sup_{\alpha\in\N^d_{0,t}} \log\frac{|t^\alpha|}{e^{\varphi^*_\omega(\lambda |\alpha|)/\lambda}}\leq \sup_{\alpha\in\N^d_{0,t}}\left\{ \lambda|\alpha|\log |t| -\varphi^*_\omega(\lambda|\alpha|)\right\} \\
&\leq& \varphi^{**}_\omega(\log |t|) = \omega(t),
\eeqsn
so the first inequality of \eqref{4a} is proved. Now, similarly to \cite[proof of Lemma 5.7]{RS}, we can prove that, for every $t\in\R^d$ such that $|t|\geq e^{\varphi^*_\omega(\lambda)/\lambda}$,
\beqs
\label{4e}
\omega(t)\leq 2\sup_{M\in\N_0}\left\{ \lambda M\log |t|-\varphi^*_\omega(\lambda M)\right\}.
\eeqs

Observe now that for every $t\in\R^d$, we have $|t|\leq \sqrt{d} |t|_\infty\leq d |t|_\infty.$ 
Then by \cite[Remark 2.2(iii)]{BJO-PW},
\beqs
\label{4f}
\omega(t)\leq\omega(d|t|_\infty)\leq D_d\left(\omega(|t|_\infty)+1\right),
\eeqs
for $D_d=L+L^2+\ldots+L^{d-1}$, where $L$ is the constant of condition $(\alpha)$ in
Definition~\ref{defomega}.

Fix now $t\in\R^d$ with $|t|\geq e^{\varphi^*_\omega(\lambda)/\lambda}$ and let $j_0$ be such that $|t|_\infty=|t_{j_0}|$; for every $M\in\N_0$, we then write $\alpha_M:=Me_{j_0}$. We then have $|t|_\infty^M=|t^{\alpha_M}|$, and so by \eqref{4e} we obtain
\beqsn
\omega(|t|_\infty)=\omega(|t_{j_0}|)\leq 2\lambda\sup_{M\in\N_0}\log \frac{|t^{\alpha_M}|}{e^{\varphi^*_\omega(\lambda |\alpha_M|)/\lambda}}\leq 2\lambda\sup_{\alpha\in\N^d_{0,t}} \log\frac{|t^\alpha|}{e^{\varphi^*_\omega(\lambda |\alpha|)/\lambda}}=2\lambda\omega_{\bfW^{(\lambda)}}(t),
\eeqsn
since $\alpha_M\in\N^d_{0,t}$ due to the fact that $t_{j_0}\neq 0$ (we are in fact considering $t\in\R^d$ such that $|t|\geq e^{\varphi^*_\omega(\lambda)/\lambda}$). By \eqref{4f} we then obtain
\beqsn
\omega(t)\leq 2\lambda D_d\omega_{\bfW^{(\lambda)}}(t)+D_d
\eeqsn
for $|t|\geq e^{\varphi^*_\omega(\lambda)/\lambda}$. Then the second inequality of \eqref{4a} holds for
\beqsn
B=2D_d\quad\text{and}\quad C_\lambda=D_d+\sup_{|t|\leq e^{\varphi^*_\omega(\lambda)/\lambda}}\omega(t).
\eeqsn
\end{proof}

\begin{Lemma}
\label{lemma51G}
Let $\omega$ be a weight function and consider the weight matrix
$\M_\omega$ as defined in \eqref{Wlambda}, \eqref{calMomega}. Then, for
$r>0$:
\begin{itemize}
\item[(a)]
$\omega(t)=O(t^{1/r})$ as $t\to+\infty$ if and only if
\beqs
\label{55G}
\forall\,\lambda>0\ \exists\, C,D\geq1\ \forall \alpha\in\N^d:\
\alpha^{r\alpha}\leq CD^{|\alpha|}W^{(\lambda)}_\alpha;
\eeqs
\item[(b)]
$\omega(t)=o(t^{1/r})$ as $t\to+\infty$ if and only if
\beqs
\label{56G}
\forall\,\lambda,D>0\ \exists\, C\geq1\ \forall \alpha\in\N^d:\
\alpha^{r\alpha}\leq CD^{|\alpha|}W^{(\lambda)}_\alpha.
\eeqs
\end{itemize}
Moreover, in the conditions above we can replace ``$\ \forall\,\lambda$" by ``$\ \exists\,\lambda$".
\end{Lemma}

\begin{proof}
We only consider the case ``$\ \forall\,\lambda$", since the proof for the case ``$\ \exists\,\lambda$" is analogous.

$(a)$: If $\omega(t)=O(t^{1/r})$ as $t\to+\infty$, there exists $c\geq1$ such that
\beqsn
\omega(t)\leq ct^{1/r}+c,\qquad t\geq0,
\eeqsn
and hence
\beqsn
\varphi_\omega(y)=\omega(e^y)\leq ce^{y/r}+c,\qquad y\geq0.
\eeqsn
Then
\beqs
\nonumber
\varphi^*_\omega(x)=&&\sup_{y\geq0}\{xy-\varphi_\omega(y)\}
\geq\sup_{y\geq0}\{xy-ce^{y/r}\}-c\\
\label{C30}
=&&xr\left(\log\frac{xr}{c}-1\right)-c,\qquad\mbox{if}\ x\geq\frac cr\,.
\eeqs
Therefore, for every $\lambda>0$ and $j\in\N$ with $j\geq\frac{c}{r\lambda}$,
choosing $x=\lambda j$ and multiplying by $1/\lambda$ in \eqref{C30}, we have
\beqsn
\frac1\lambda\varphi^*_\omega(\lambda j)\geq
jr\left(\log\frac{\lambda jr}{c}-1\right)-\frac c\lambda
=\log j^{jr}+jr\log\frac{\lambda r}{ec}-\frac c\lambda
\eeqsn
 and hence, for $j\geq\frac{c}{r\lambda}$,
 \beqs
 \label{4g}
 j^{jr}\leq e^{\frac{1}{\lambda}\varphi^*_\omega(\lambda j)}
 \left(\frac{ec}{\lambda r}\right)^{jr}e^{\frac c\lambda}
 \leq \tilde{C}_\lambda D_\lambda^j\tilde{W}^{(\lambda)}_j
 \eeqs
 for $\tilde{C}_\lambda=e^{c/\lambda}$, $D_\lambda=\max\left\{\left(
 \frac{ec}{\lambda r}\right)^r,1\right\}$, and $\tilde{W}^{(\lambda)}_j=e^{\varphi^*_\omega(\lambda j)/\lambda}$.
 Enlarging the constants $\tilde{C}_\lambda,D_\lambda$ we have \eqref{4g} for all
 $j\in\N$. Then,
 \beqsn
 \alpha^{r\alpha}=\alpha_1^{r\alpha_1}\dots\alpha_d^{r\alpha_d}\leq \tilde{C}_\lambda D_\lambda^{\alpha_1}\tilde{W}^{(\lambda)}_{\alpha_1}\dots \tilde{C}_\lambda D_\lambda^{\alpha_d}\tilde{W}^{(\lambda)}_{\alpha_d},
 \eeqsn
 and so we obtain \eqref{55G} for $C=\tilde{C}_\lambda^d$ in view of Lemma~\ref{fromRS}(vii).

 Conversely, if \eqref{55G} holds then, by definition of associated function we obtain, for $z\in\R^d$,
 \beqsn
 \omega_{\bfW^{(\lambda)}}(z)=\sup_{\alpha\in\N^d_{0,z}}\log\frac{|z^\alpha|}{W^{(\lambda)}_\alpha}
 \leq\sup_{\alpha\in\N^d_{0,z}}\log |z^\alpha|\frac{CD^{|\alpha|}}{\alpha^{r\alpha}}\leq
 \sup_{\alpha\in\N^d_{0,z}}\left(\log C+\sum_{j=1}^d \log\frac{(|z_j| D)^{\alpha_j}}{\alpha_j^{r\alpha_j}}\right).
 \eeqsn
 Consider now $j$ such that $z_j\neq 0$ (otherwise the corresponding addend in the previous sum is $0$). A simple computation shows that
 \beqsn
 \sup_{\alpha_j\in \N}\log\frac{(|z_j| D)^{\alpha_j}}{\alpha_j^{r\alpha_j}}\leq \sup_{s>0}\log\frac{(|z_j| D)^{s}}{s^{rs}}\leq \frac{r}{e}(|z_j| D)^{1/r}.
 \eeqsn
 We then have
 \beqs
 \label{C31}
 \omega_{\bfW^{(\lambda)}}(z)\leq \log C+\sum_{j=1}^d \frac{r}{e}(|z_j| D)^{1/r}\leq \log C+\frac{dr}{e}(|z| D)^{1/r}.
 \eeqs
 By Lemma~\ref{lemma52G},  we have $\omega(z)=\omega(|z|)=O(|z|^{1/r})$ as 
 $|z|\to+\infty$ for $z\in\R^d$,
 which is equivalent to $\omega(t)=O(t^{1/r})$ as $t\to+\infty$ for $t\in\R$.

 $(b)$:
 If $\omega(t)=o(t^{1/r})$ as $t\to+\infty$, then for every $D>0$ there exists $c>0$
 such that
 \beqsn
 \omega(t)\leq Dt^{1/r}+c,\quad t\geq0.
 \eeqsn
 Proceeding as in $(a)$ we have 
 \beqsn
 \varphi^*_\omega(x)\geq xr\left(\log\frac{xr}{D}-1\right)-c,
 \qquad\mbox{for}\ x\geq\frac D r,
 \eeqsn
and hence
\beqsn
\alpha^{r\alpha}\leq e^{c/\lambda}\left(\frac{eD}{\lambda r}\right)^{r|\alpha|}W^{(\lambda)}_\alpha
\eeqsn
and \eqref{56G} is satisfied by the arbitrariness of $D>0$.

Conversely, if \eqref{56G} holds then, proceeding as in $(a)$, we have that for
every $\lambda,D>0$ there exists $C>0$ such that \eqref{C31} is valid and therefore,
by Lemma~\ref{lemma52G}, $\omega(z)=o(|z|^{1/r})$ as  $|z|\to+\infty$ for
$z\in\R^d$,
or, equivalently, $\omega(t)=o(t^{1/r})$ as  $t\to+\infty$.
\end{proof}

\begin{Cor}
\label{remopiccolo}
Let $\omega$ be a weight function. We have:
\begin{enumerate}
\item[(a)] The Hermite functions belong to  $\Sch_{\{\omega\}}(\R^d)$ if and only if $\omega(t)=O(t^2)$ as $t\to+\infty$.

\item[(b)] The Hermite functions belong to   $\Sch_{(\omega)}(\R^d)$ if and only if $\omega(t)=o(t^2)$ as $t\to+\infty$.
\end{enumerate}
\end{Cor}
\begin{proof}
By Lemmas~\ref{lemma51G} and \ref{fromRS} and Proposition~\ref{lemma37G}, $\omega(t)=O(t^2)$ as $t\to+\infty$ if and only if  $\M_\omega$ satisfies \eqref{12L2R} if and only if the space $\Sch_{\{\M_\omega\}}$ contains the Hermite functions; while
$\omega(t)=o(t^2)$ as $t\to+\infty$ if and only if $\M_\omega$ satisfies \eqref{12L2B} if and only if  $\Sch_{(\M_\omega)}$ contains the Hermite functions.
\end{proof}

For a weight function $\omega$ we now consider the sequence spaces
\beqsn
&&\Lambda_{\{\omega\}}:=\{\bfc=(c_\alpha)\in \C^{\N_0^d}:\ \exists\, j\in\N,\ \
\|\bfc\|_{\omega,j}:=\sup_{\alpha\in\N^d_0}|c_\alpha|e^{\frac{1}{j}\omega(\alpha^{1/2}/ j)}
<+\infty\},\\
&&\Lambda_{(\omega)}:=\{\bfc=(c_\alpha)\in \C^{\N_0^d}:\ \forall\,j\in\N,\
\
\|\bfc\|_{\omega,1/j}=\sup_{\alpha\in\N^d_0}|c_\alpha|e^{j\omega(\alpha^{1/2} j)}<+\infty\}.
\eeqsn

\begin{Prop}
\label{lemma53G}
Let $\omega$ be a weight function and $\M_\omega$ the weight matrix defined by
\eqref{Wlambda}, \eqref{calMomega}. Then $\Lambda_{\{\omega\}}=\Lambda_{\{\M_\omega\}}$ and $\Lambda_{(\omega)}=\Lambda_{(\M_\omega)}$
and the equalities are also topological.
\end{Prop}

\begin{proof}
From Lemma~\ref{lemma52G} with $\lambda=j$ (and taking $B\in\N$), we have
\beqsn
e^{\frac{1}{Bj}\omega\left(\frac{\alpha^{1/2}}{Bj}\right)}
\leq
e^{\omega_{\bfW^{(j)}}\left(\frac{\alpha^{1/2}}{Bj}\right)+\frac{C_j}{Bj}}
\leq e^{\frac{C_j}{Bj}}e^{\omega_{\bfW^{(j)}}({\alpha^{1/2}}/{j})}
\eeqsn
and, conversely, $e^{\omega_{\bfW^{(j)}}({\alpha^{1/2}}/{j})}\leq e^{\frac1j\omega(\alpha^{1/2}/j)}.$ 
This proves the Roumieu case. Taking 
$\lambda=1/j$ we prove analogously the Beurling case.
\end{proof}
We now easily deduce the following consequence of Theorem~\ref{th41G}.
\begin{Cor}
\label{th54G}
Let $\omega$ be a weight function.
The Hermite
functions are an absolute Schauder basis in $\Sch_{\{\omega\}}(\R^d)$ and
\beqsn
T:\ \Sch_{\{\omega\}}(\R^d)&&\longrightarrow \Lambda_{\{\omega\}}\\
f&&\longmapsto (\xi_\gamma(f))_{\gamma\in\N_0}
\eeqsn
defines an isomorphism.

If moreover $\omega(t)=o(t^2)$ as $t\to+\infty$, then
the Hermite
functions are an absolute Schauder basis  in $\Sch_{(\omega)}(\R^d)$ and
\beqsn
T:\ \Sch_{(\omega)}(\R^d)\longrightarrow \Lambda_{(\omega)}
\eeqsn
as defined above is also an isomorphism.
\end{Cor}

We finally have
\begin{Cor}
\label{cor56G}
If $\omega$ is a weight function, then
$\Sch_{\{\omega\}}(\R^d)$ is nuclear. If moreover $\omega(t)=o(t^2)$ as
$t\to+\infty$, then
$\Sch_{(\omega)}(\R^d)$ is nuclear. 
\end{Cor}




\vskip\baselineskip
{\bf Acknowledgments.}
The first three authors were partially supported by  the Project FFABR 2017 (MIUR),
and by the Projects FIR 2018, FAR 2018 and FAR 2019 (University of Ferrara).
The first and third  authors are members of the Gruppo Nazionale per l'Analisi
Matematica, la Probabilit\`a e le loro Applicazioni (GNAMPA) of the Istituto Nazionale di Alta
Matematica (INdAM). The research of the second author was partially supported by the project MTM2016-76647-P and the grant BEST/2019/172 from Generalitat Valenciana. The fourth author is supported by FWF-project J 3948-N35 and FWF-project P32905.


\end{document}